\documentclass[11pt]{article}
\usepackage[utf8]{inputenc}
\usepackage{amsmath, amsthm, amssymb,fullpage, xcolor}
\usepackage{bbm}
\usepackage{boxedminipage}
\usepackage{libertine}
\usepackage{libertinust1math}
\usepackage[T1]{fontenc}
\usepackage[font=small,labelfont=bf]{caption}
\usepackage{hyperref}
\usepackage{microtype}
\usepackage[chicago]{xellipsis}
\usepackage{pgf,tikz,pgfplots}
\pgfplotsset{compat=1.14}
\usepackage{mathrsfs}
\usepackage{enumitem}
\usepackage{hyperref}
\usepackage{cleveref}
\usepackage{booktabs}
\usepackage{marginnote}
\usepackage{threeparttable}

\usetikzlibrary{arrows}
\hypersetup{
    colorlinks=true, 
    citecolor=blue, 
    linkcolor=blue, 
    urlcolor=blue, 
    linktoc=all 
}

\newtheorem{theorem}{Theorem}[section]
\newtheorem{lemma}[theorem]{Lemma}
\newtheorem{corollary}[theorem]{Corollary}
\newtheorem{proposition}[theorem]{Proposition}
\theoremstyle{definition}
\newtheorem{definition}[theorem]{Definition}
\newtheorem{observation}[theorem]{Observation}
\newtheorem{remark}[theorem]{Remark}
\newtheorem{problem}[theorem]{Problem}

\newtheorem{conjecture}[theorem]{Conjecture}

\newtheorem{assumption}{Assumption}

\theoremstyle{theorem}
\providecommand{\gaussiangenericname}{}
\newtheorem{innergaussiangeneric}{\gaussiangenericname}

\newcommand{\makegaussian}[2]{%
  \newenvironment{#1}[1]
  {%
   \renewcommand\gaussiangenericname{#2}%
   \renewcommand\theinnergaussiangeneric{{##1}G}%
   \innergaussiangeneric
  }
  {\endinnergaussiangeneric}
}
\makegaussian{gtheorem}{Theorem}
\makegaussian{glemma}{Lemma}
\makegaussian{gproposition}{Proposition}

\providecommand{\restategenericname}{}
\newtheorem{innerrestategeneric}{Restatement of \restategenericname}
\newcommand{\makerestate}[2]{%
  \newenvironment{#1}[1]
  {%
    \renewcommand\restategenericname{#2}%
   \renewcommand\theinnerrestategeneric{{##1}}%
   \innerrestategeneric
  }
  {\endinnerrestategeneric}
}
\makerestate{rtheorem}{Theorem}
\makerestate{rlemma}{Lemma}
\makerestate{rproposition}{Proposition}

\newcommand{\poly}{\mathrm{poly}}

\newcommand{\im}{\Im}
\newcommand{\defeq}{:=}

\newcommand{\E}{\mathbb{E}}

\newcommand{\R}{\mathbb{R}}
\newcommand{\C}{\mathbb{C}}

\renewcommand{\P}{\mathbb{P}}

\newcommand{\vol}{\mathrm{Leb}}
\newcommand{\dE}{\mathbb{E}}

\newcommand{\eps}{\varepsilon}

\newcommand{\Tr}{\mathrm{Tr}}

\newcommand{\gap}{\mathrm{gap}}

\newcommand{\T}{\intercal}
\newcommand{\calN}{\mathcal{N}}

\newcommand{\leb}{\mathrm{Leb}}

\newcommand{\rank}{\mathrm{rank}}

\newcommand{\crv}{C_{\mathrm{RV}}}
\newcommand{\density}{\delta_\infty}

\newcommand{\expbound}{2e^{-2 n}}

\newcommand{\md}{\,\middle|\,}
\newcommand{\indicator}[1]{\left\{{#1}\right\}}

\newcommand{\blambda}{\boldsymbol{\lambda}}

\newcommand{\bnu}{\boldsymbol{\nu}}

\newcommand{\bm}{\boldsymbol{m}}

\newcommand{\bx}{{\boldsymbol{x}}}
\newcommand{\by}{\boldsymbol{y}}
\newcommand{\bz}{\boldsymbol{z}}
\newcommand{\bA}{\boldsymbol{A}}

\newcommand{\bE}{\boldsymbol{E}}

\newcommand{\bG}{\boldsymbol{G}}
\newcommand{\bH}{\boldsymbol{H}}

\newcommand{\bM}{\boldsymbol{M}}
\newcommand{\bN}{\boldsymbol{N}}

\newcommand{\bS}{\boldsymbol{S}}

\newcommand{\bW}{\boldsymbol{W}}
\newcommand{\bX}{\boldsymbol{X}}
\newcommand{\bY}{\boldsymbol{Y}}
\newcommand{\bZ}{\boldsymbol{Z}}

\title{Overlaps, Eigenvalue Gaps, and Pseudospectrum under real Ginibre and Absolutely Continuous Perturbations}
\author{Jess Banks \thanks{Supported by  the NSF Graduate Research Fellowship Program under Grant DGE-1752814.}\\ jess.m.banks@berkeley.edu \\ UC Berkeley \and Jorge Garza-Vargas\\jgarzavargas@berkeley.edu\\ UC Berkeley \and  Archit Kulkarni \thanks{Supported by a James H. Simons Fellowship.} \\ akulkarni@berkeley.edu\\ UC Berkeley \and Nikhil Srivastava \thanks{Supported by NSF Grant CCF-1553751.} \\ nikhil@math.berkeley.edu \\  UC Berkeley }
\date{\today}
\begin{document}

\maketitle
\begin{abstract}
	Let $G_n$ be an $n \times n$ matrix with real i.i.d. $N(0,1/n)$ entries, let $A$ be a real $n \times n$ matrix with $\Vert A \Vert \le 1$, and let $\gamma \in (0,1)$.  We show that with probability $0.99$, $A + \gamma G_n$ has all of its eigenvalue condition numbers bounded by $O\left(n^{5/2}/\gamma^{3/2}\right)$ and eigenvector condition number bounded by $O\left(n^3 /\gamma^{3/2}\right)$.  Furthermore, we show that for any $s > 0$, the probability that $A + \gamma G_n$ has two eigenvalues within distance at most $s$ of each other is $O\left(n^4 s^{1/3}/\gamma^{5/2}\right).$ In fact, we show the above statements hold in the more general setting of non-Gaussian perturbations  with real, independent, absolutely continuous entries with a finite moment assumption and appropriate normalization. 
	
	This extends the previous work \cite{banks2019gaussian} which proved an eigenvector condition number bound of $O\left(n^{3/2} / \gamma\right)$ for the simpler case of \emph{complex} i.i.d. Gaussian matrix perturbations.
	The case of real perturbations introduces several challenges stemming from the weaker anticoncentration properties of real vs. complex random variables. A key ingredient in our proof is new lower tail bounds on the small singular values of the complex shifts $z-(A+\gamma G_n)$ which recover the tail behavior of the  complex Ginibre ensemble when $\Im z\neq 0$. This yields sharp control on the area of the pseudospectrum $\Lambda_\eps(A+\gamma G_n)$ in terms of the pseudospectral parameter $\eps>0$, which is sufficient to bound the overlaps and eigenvector condition number via a limiting argument.
	
\end{abstract}

\section{Introduction}
 The stability of the eigenvalues of a matrix under perturbations is a central issue in numerical analysis, random matrix theory, control theory, and other areas of mathematics. Suppose 
 $$X=\sum_{i=1}^n \lambda_i {v_i w_i^*}$$ is a diagonalizable matrix with distinct eigenvalues $\lambda_1,\ldots,\lambda_n\in\C$ and left and right eigenvectors $\{w_i^*, v_i\}_{i=1}^n$ normalized so that $w_i^*v_i=1$. Then the speed at which $\lambda_i$ moves under an arbitrary perturbation is governed by its {\em eigenvalue condition number} (also called {\em overlap}\footnote{ The $n\times n$ \textit{overlap matrix} of $X$ is defined as $\mathscr{O}(X)_{i,j} = v_j^\ast v_i \overline{w_j^\ast w_i}$, so that $\mathscr{O}(X)_{i,i} = \kappa(\lambda_i)^2.$} in the mathematical physics literature), defined as:
$$
    \kappa(\lambda_i) \defeq \|v_i\| \|w_i\|,\quad i=1,\ldots,n.
$$
All of the eigenvalue condition numbers are equal to one for normal $X$, and approach $\infty$ as $X$ approaches a non-diagonalizable matrix, thus constituting a quantitative measure of nonnormality.

In this paper, we study the extent to which adding a small {\em real} random matrix to an arbitrary real matrix tames its eigenvalue condition numbers and other related parameters. 
Specifically, we consider random matrices of type $$X=A+\gamma \bM$$ 
where $A\in\R^{n\times n}$ is an arbitary deterministic matrix with $\|A\|\le 1$, $\gamma$ is a real parameter, and $\bM\in\R^{n\times n}$ is either a normalized {\em real Ginibre matrix} (i.e., with i.i.d. $N(0,1/n)$ entries) or more generically any random matrix with independent, real absolutely continuous entries. 
Our main result (Theorems \ref{thm:kappai-probabilistic-intro} and \ref{thm:gaussiankappai-probabilistic-intro}) is that with high probability, such an $X$ has {\em all} of its $\kappa(\lambda_i)$ bounded by a small polynomial in $n/\gamma$. Previously, such a result was only known in the case of {\em complex} Ginibre perturbations \cite{banks2019gaussian}, and no bound was known for all of the overlaps even in the special case of a centered real Ginibre matrix.\footnote{i.e., the case $A=0$; in this setting \cite{fyodorov2018statistics} derived very precise bounds for the overlaps corresponding to the real eigenvalues only.} Our results straightforwardly imply similar bounds on the {\em eigenvector condition number} of $X$, another measure of nonnormality of particular interest in numerical linear algebra.

As in \cite{banks2019gaussian}, the proofs of our theorems rely on studying the {\em $\eps-$pseudospectrum} of $X$, defined as:
\begin{align*}
    \Lambda_\eps(X) 
    &= \left\{z \in \C : \sigma_n(z - X) \le \eps \right\},
\end{align*}
where $\sigma_1 \ge \cdots \ge \sigma_n$ denote the singular values of a matrix in descending order. The connection between the pseudospectrum and the $\kappa(\lambda_i)$ is given by the elementary limiting formulas:
\begin{equation}\label{eqn:introrealformula}
2\sum_{\lambda_i\in\R}\kappa(\lambda_i) = \lim_{\eps\rightarrow 0} \frac{\leb_{\R}(\Lambda_\eps(X)\cap\R)}{\eps}
\end{equation}
and
\begin{equation}\label{eqn:introcomplexformula}
\pi\sum_{\lambda_i\in\C\setminus\R}\kappa(\lambda_i)^2 = \lim_{\eps\rightarrow 0} \frac{\leb_{\C}(\Lambda_\eps(X)\cap(\C\setminus \R))}{\eps^2},
\end{equation}
which may be proven by examining the spectral expansion of the resolvent. After some judicious switching of limits and integrals, the right hand sides can be controlled by obtaining bounds on the probabilities
\begin{equation}\label{eqn:lambdasketch}\P\{z\in \Lambda_\eps(X)\} = \P\{\sigma_n(z-X)\le \eps\}\end{equation}
for shifts $z\in\C$, {\em provided one obtains  the correct exponent} $\eps^1$ for $z\in\R$ and $\eps^2$ for $z\in\C\setminus\R$.

The pursuit of such bounds is the main technical theme of the paper, in contrast to much of the rest of random matrix theory where the emphasis is on obtaining sharp dependence on $n$. Specifically, our main probabilistic result (Theorems \ref{thm:singvalscomplexshifts-intro} and \ref{prop:gaussiansingvalscomplex-intro}) shows that the probability in \eqref{eqn:lambdasketch} can always be taken to be $O(\eps)$ for $z\in\R$ and $O(\eps^2/|\Im(z)|)$ for $z\notin \R$, which is good enough to take the limit as $\eps\rightarrow 0$ after establishing that there are unlikely to be eigenvalues of $X$ near the real line but not on it.  Note that real Ginibre matrices are known to have $\Theta(\sqrt{n})$ real eigenvalues on average \cite{edelman1994many}, so one cannot ignore the eigenvalues on the real line. We also develop tail bounds with the correct $\epsilon^{k^2}$ and $\epsilon^{2k^2}$ scaling for the $k$th smallest singular values of real and complex shifts of $X$, for $k=O(n)$ when $\bM$ is a real Ginibre matrix and $k=O(\sqrt{n})$ for more general $\bM$.

A secondary contribution of the paper, which also plays a role in the proof above, is a polynomial (in $\gamma/n$) lower bound on the {\em minimum eigenvalue gap}:
$$\gap(X):=\min_{i\neq j}|\lambda_i-\lambda_j|$$
 which holds with high probability (Theorem \ref{thm:gaps}). The novelty of this result in comparison to existing minimum gap bounds (such as \cite{ge2017eigenvalue, luh2020eigenvectors}) is that it works for {\em heterogeneous non-centered} random matrices $X$, as opposed to only matrices with i.i.d. entries. This noncenteredness is crucial to applications in numerical linear algebra, where a random perturbation is used to regularize the eigenvalue gaps of an arbitrary input matrix as in \cite{banks2019pseudospectral}. The minimum gap proof relies on controlling the two smallest singular values of real and complex shifts of $X$, for which we employ Theorems \ref{thm:singvalscomplexshifts-intro} and \ref{prop:gaussiansingvalscomplex-intro}.
 
We now proceed with a formal statement of our results and detailed discussion of related work.

\begin{remark}[Concurrent and Independent Work] After completing this manuscript, we learned of the independent work \cite{jain20} which obtains results similar to ours regarding the eigenvector condition number and minimum eigenvalue gap. Their bound on $\kappa_V$ improves Theorem \ref{thm:kappai-probabilistic-intro} by a factor of $O(n/(\sqrt{\gamma}\log(n/\gamma)))$, thus almost matching the dependence on $\gamma$ in Davies' conjecture \cite{davies};  their bound on the minimum eigenvalue gap is also better than that supplied by Theorem \ref{thm:gaps} by a $\poly(n/\gamma)$ factor. They do not obtain specific control on the  $\kappa(\lambda_i)$ for real and complex $\lambda_i$ separately, and our bound for the sum of the real $\kappa(\lambda_i)$ in Theorem \ref{thm:kappai-probabilistic-intro} implies a bound for the maximum which is slightly better than their $\kappa_V$ bound alone.

The techniques used by both papers focus on deriving tail bounds for the least singular value with the correct scaling in $\eps$, but the proofs are essentially different. In particular, our proof relies on studying the entries of the resolvent, whereas theirs is more geometric. We obtain bounds on the $k$th smallest singular values of real and complex shifts (Theorems \ref{thm:skboundgeneral}--\ref{prop:gaussiansingvalscomplex-intro}) with the correct $\eps^{k^2}$ and $\eps^{2k^2}$ scaling, whereas they derive bounds for $k=1,2$, but with better dependence on $n$.

They do not take the limit as $\eps\rightarrow 0$ to derive $\kappa_V$ bounds, relying instead on a bootstrapping scheme, while we do.
\end{remark}

\subsection{Results and Organization}
\emph{Notation.} We use boldface to denote random quantities.  For any $n \in \mathbb{N}$, we use the shorthand $[n]$ to denote the set $\{1, 2, \dots, n\}$.  For a matrix $X$, we let $\Vert X \Vert$ denote its spectral norm. 

Throughout the paper, we will write $\bM_n$ for an $n\times n$ real random matrix satisfying the following assumption:
\begin{assumption}
    \label{assumption}
    The matrix $\bM_n$ has independent entries, each with density on $\R$ bounded almost everywhere by $\sqrt n K>0$. 
\end{assumption}

\noindent Equivalently, $\bM_n = n^{-1/2}\widehat{\bM}_n$ where $\widehat{\bM}_n$ has independent real entries with density bounded by $K$. We do not require that $\bM_n$ have mean zero, nor will we make any explicit moment assumptions on its entries. Instead, our results will often be stated in terms of the $L_p$ norm of its operator norm, which we denote by
\begin{equation}
    \label{eq:cml-def}
    B_{\bM_n,p} \defeq \E \left[\|\bM_n\|^p \right]^{1/p}.
\end{equation}

\begin{remark} Depending on the entry distributions of $\bM_n$, the $\sqrt n$ in Assumption \ref{assumption} need not be the appropriate normalization so that $\E\|\bM_n\| = O(1)$. However, this holds in the case when the entries of $\bM_n$ have bounded fourth moment, and we include this explicit scaling for easier comparison to the Gaussian case.\end{remark}

\begin{definition}
    We will write $\bG_n$ to denote a \textit{normalized real Ginibre matrix}. In other words, the entries of $\bG_n$ are independent real random variables, each distributed as $\calN(0,1/n)$.
\end{definition}
\noindent Of course, $\bG_n$ satisfies Assumption \ref{assumption} with $K = 1/\sqrt{2\pi}$.

We can now state our main theorems. We begin with the singular value tail bounds that are the probabilistic workhorse of the paper. Although we only use the bounds for the bottom two singular values corresponding to $k=1,2$, we state the bounds for general $k$ as we are able to obtain the optimal $\eps^{k^2}$ and $\eps^{2k^2}$ type dependence matching the centered real Ginibre and complex Ginibre cases (cf. Theorem \ref{thm:szarek}). We state each bound twice---first for matrices satisfying Assumption \ref{assumption}, and then specialized to real Ginibre  perturbations, for which we are able to obtain improvements by exploiting specific properties of Gaussians.

\begin{theorem}[Singular Values of $\bM_n$] \label{thm:skboundgeneral}
Let  $\bM_n \in \mathbb{R}^{n \times n}$ be a random matrix satisfying Assumption~\ref{assumption} with parameter $K>0$.  Then
$$\P\left[\sigma_{n-k+1}(\bM_n) \leq \eps \right]  \le \binom{n}{k} \left(\crv K \eps \sqrt{kn(n - k + 1)} \right)^{k^2} \le n^{k^2 + k} k^{\frac{1}{2} k^2}(\crv K)^{k^2} \eps^{k^2},$$
where $\crv$ is the universal constant appearing in Theorem \ref{thm:rvproj}, due to \cite{rudelson2015small}.
\end{theorem}
Note that Theorem \ref{thm:skboundgeneral} includes as a special case matrices of type $z-(A+\gamma \bM_n)$ for real $z$ and $A$, as such matrices themselves satisfy Assumption 1.

\begin{gtheorem}{\ref{thm:skboundgeneral}}[Singular Values of Real Shifts: Gaussian]  \label{thm:gaussianrealrealsingular}
    Let $z \in \R$ and $A \in \mathbb{R}^{n \times n}$ be deterministic, and let $\bG_n$ be a normalized Ginibre matrix. For every $\gamma > 0$,
    \[
        \P[\sigma_{n-k+1}(z - (A + \gamma \bG_n)) \le \eps] \le \left(\frac{\sqrt{2e} n \eps}{k \gamma} \right)^{k^2}.
    \]
    In the case $k = 1$, one has a better constant:
    \[ \P[\sigma_{n}(z - (A + \gamma \bG_n)) \le \eps] \le \frac{n \eps}{\gamma}.\]
\end{gtheorem}
The key improvement we obtain the case of nonreal complex $z$ is an extra factor of $2$ in the exponent.
\begin{theorem}[Singular Values of Complex Shifts]
\label{thm:singvalscomplexshifts-intro}
    Let $z\in \C\setminus \R$ and $A\in\R^{n\times n}$ be deterministic, and let $\bM_n$ satisfy Assumption \ref{assumption} with parameter $K>0$. For every $k \le \sqrt n - 2$,
    $$
        \P\left[\sigma_{n-k+1}\left(z - (A + \bM_n)\right) \le \eps \right] \le (1 + k^2){\binom{n}{k}}^2\left(C_{\ref{thm:singvalscomplexshifts-intro}} k^2(nK)^3\left(\left( B_{\bM_n,2k^2} + \|A\| + |\Re z|\right)^2 + |\Im z|^2 \right)\frac{\eps^2}{|\Im z|}\right)^{k^2},
    $$
    where $C_{\ref{thm:singvalscomplexshifts-intro}}$ is a universal constant defined in \eqref{eq:singvalscomplexshifts-const}.
\end{theorem}

\begin{gtheorem}{\ref{thm:singvalscomplexshifts-intro}}[Singular Values of Complex Shifts: Gaussian]
\label{prop:gaussiansingvalscomplex-intro}
    Let $z\in\C \setminus \R$ and $A \in \R^{n\times n}$ be deterministic, and let $\bG_n$ be a normalized $n\times n$ real Ginibre matrix. For every $\gamma > 0$, and every $k \le n/7$, 
    $$
        \P\left[\sigma_{n-k+1}(z - (A + \gamma \bG_n)) \le \eps \right] \le {\binom{n}{k}}^2 \left(\frac{\sqrt{7e} k^2 n^3}{2\gamma^3} \left(\left(9\gamma  + \|A\| + |\Re z|\right)^2 + |\Im z|^2\right)\frac{\eps^2}{|\Im z|}\right)^{k^2}.
    $$
\end{gtheorem}

The proofs of Theorems \ref{thm:skboundgeneral}--\ref{prop:gaussiansingvalscomplex-intro} appear in Sections \ref{sec:realreal} and \ref{sec:realcomplex}, and rely on anticoncentration bounds for quadratic polynomials in independent, absolutely continuous random variables as well as Gaussians, which may be of independent interest and are developed in Section \ref{sec:anticoncentration}.

Using the above theorems to control the bottom two singular values of real and complex shifts of $X$, and employing simple net arguments in the complex plane, we obtain the following minimum gap bounds in Section \ref{sec:gaps}.
\begin{theorem}[Minimum Eigenvalue Gap]     \label{thm:gaps}
    Let $n\geq 16$, $A \in \R^{n\times n}$ be deterministic, and $\bM_n$ be a random matrix satisfying Assumption \ref{assumption} with parameter $K>0$. For any $0 < \gamma < K$ and $R> 1$:
    \begin{equation} \label{eqn:mingaptheorem}
        \P\left[\gap(A + \gamma \bM_n) \le s \right] \le C_{\ref{thm:gaps}} R^2\left(\gamma B_{\bM_n,8} + \|A\| + R\right) (K/\gamma)^{5/2} n^4 s^{1/3} + \P\left[\|A + \bM_n\| \ge R\right],
    \end{equation}
    where $C_{\ref{thm:gaps}}$ is a universal constant defined in equation \eqref{eq:gaps-const-def}. Moreover, if $\bG_n$ is an $n\times n$ real Ginibre and $0 < \gamma < 1$ then
    \begin{equation} \label{eqn:gaussianmingaptheorem}
        \P\left[\gap(A + \gamma \bG_n) \le s \right] \le  15 \left(\|A\| + 7\right)^3 n^3\gamma^{-5/2}s^{1/3} + e^{-2n}.
    \end{equation}
\end{theorem}

Finally, by combining all of the above theorems and carrying out the pseudospectral area approach outlined in the introduction, we obtain the advertised results regarding eigenvalue condition numbers and eigenvector condition numbers.
Recall that the eigenvector condition number of a matrix is defined as:
$$
    \kappa_V(X) \defeq \inf \left\{\|V\|\|V^{-1}\| : X = VDV^{-1} \text{ for diagonal }D \right\},
$$
with $\kappa_V(X)\defeq\infty$ if $X$ is not diagonalizable.
In the following theorem, a typical setting has $\Vert A \Vert, \Vert M_n \Vert, K, $ and $R$ all of order $\Theta(1)$, so one may obtain upper bounds of order $\poly(n, 1/\gamma)$ with high probability by setting $\eps_1, \eps_2$ appropriately.
\begin{theorem}[Eigenvalue and Eigenvector Condition Numbers]\label{thm:kappai-probabilistic-intro}
Let $n \ge 9$.  Let $A \in \R^{n\times n}$ be deterministic, and let $\bM_n$ satisfy Assumption \ref{assumption} with parameter $K>0$. Let $0 < \gamma < K \min\{1, \Vert A \Vert + R\}$, and write $\blambda_1,...,\blambda_n$ for the eigenvalues of $A + \gamma \bM_n$.  Let $R > \E \Vert \gamma \bM_n \Vert$.  Then for any $\eps_1, \eps_2 > 0$, with probability at least $$1 - 2\eps_1 -  O\left(\frac{R(R + \Vert A \Vert)^{3/5}K^{8/5}n^{14/5} \eps_2^{3/5}}{\gamma^{8/5}}\right) - 2 \P[\gamma \Vert \bM_n \Vert > R],$$ we have

    \[ 
        \sum_{\blambda_i \in \R} \kappa(\blambda_i) \le \eps_1^{-1} C_{\ref{thm:kappai-probabilistic-intro}}  K n^2 \frac{\Vert A \Vert + R}{\gamma}, 
    \]
    \[ 
        \sum_{\blambda_i \in \C \setminus \R} \kappa(\blambda_i)^2 \le \eps_1^{-1} \log(1/\eps_2) C_{\ref{thm:kappai-probabilistic-intro}}  K^3 n^{5} \cdot \frac{(\Vert A \Vert + R)^3}{\gamma^3},\qquad\text{and}
    \]
    \[
        \kappa_V(A + \gamma \bM_n) \le \eps_1^{-1} \sqrt{\log(1/\eps_2)} C_{\ref{thm:kappai-probabilistic-intro}}  K^{3/2} n^3 \cdot \frac{(\Vert A \Vert + R)^{3/2}}{\gamma^{3/2}},
    \]
\end{theorem}

\begin{gtheorem}{\ref{thm:kappai-probabilistic-intro}}[Eigenvalue and Eigenvector Condition Numbers: Gaussian]
\label{thm:gaussiankappai-probabilistic-intro}

Let $n \ge 7$.  Let $A \in \R^{n\times n}$ be deterministic, and let $\bG_n$ be a real Ginibre matrix. Let $0 < \gamma < \min\{1, \Vert A \Vert\}$, and write $\blambda_1,...,\blambda_n$ for the eigenvalues of $A + \gamma \bG_n$.  Then for any $\eps_1, \eps_2 > 0$, with probability at least $1 - 2 \eps_1 -  \frac{30 \|A\|^{8/5} n^{8/5}}{\gamma^{8/5}}\eps_2^{3/5} - 2 e^{-2n}$ we have
 \[ 
        \sum_{\blambda_i \in \R}  \kappa(\blambda_i) \le 5 \eps_1^{-1} n \frac{\Vert A \Vert}{\gamma},
    \]
    \[
        \sum_{\blambda_i \in \C \setminus \R} \kappa(\blambda_i)^2 \le
    1000  \eps_1^{-1} \log(1/\eps_2)\frac{n^5  \Vert A \Vert^3}{\gamma^3},\qquad\text{and}
    \]
    \[
        \kappa_V(A + \gamma \bM_n) \le  1000 \eps_1^{-1}\sqrt{\log(1/\eps_2)}\frac{n^3 \Vert A \Vert^{3/2}}{\gamma^{3/2}}. 
    \]

\end{gtheorem}
\noindent By assuming a smaller upper bound on $\gamma$, one can make order of magnitude improvements in the constants, so we have made no effort to optimize them.

\begin{remark}[Moments of Overlaps] It is known (see \cite{fyodorov2018statistics} and the discussion following Remark 2.2 there) that for the real Ginibre ensemble the expected sum of the real overlaps
$$\E \sum_{\blambda_i\in\R}\mathscr{O}_{ii} = \E \sum_{\blambda_i\in \R} \kappa(\blambda_i)^2$$
is not finite. Our proof entirely avoids this divergence by working with the $\kappa(\blambda_i)$ instead of their squares. This also indicates that one should not hope to improve the power of $\eps_1$ in the first equation above to better than $-{1/2}$. 
\end{remark}
The proofs of Theorems \ref{thm:kappai-probabilistic-intro} and \ref{thm:gaussiankappai-probabilistic-intro} appear in Section \ref{sec:regularization}. We conclude with a discussion of open questions in Section \ref{sec:conclusion}.
\subsection{Related Work}

\paragraph{Eigenvalue Condition Numbers and Overlaps.} For complex Ginibre matrices, much is known about diagonal overlaps ($= \kappa(\blambda_i)^2$) and off-diagonal overlaps.  In the seminal work of Chalker and Mehlig \cite{chalker1998eigenvector} explicit formulas were given for the limiting expected overlaps as $n \to \infty$, conditioned on the locations of the participating eigenvalues.  Since then there has been significant progress; here we mention a few recent milestones.  In \cite{bourgade2018distribution}, a formula for the limiting distribution of the diagonal overlaps was proved, as well as asymptotic formulas for the expected value of all overlaps, and for correlations between overlaps.  Using a different approach, in \cite{fyodorov2018statistics}, an explicit nonasymptotic formula for the joint density of an eigenvalue and its diagonal overlap was proved.  

For the real Ginibre ensemble, results are more limited.  The same paper \cite{fyodorov2018statistics} gives an analogous joint density formula for real Ginibre matrices, but only for real eigenvalues.\footnote{Fyodorov \cite{fyodorov2018statistics} writes: "The approach suggested in the present paper can be certainly adjusted for addressing
overlaps of left/right eigenvectors corresponding to complex eigenvalues of the real Ginibre
ensemble, although in this way one encounters a few challenging technical problems not
yet fully resolved."}  Compared to a joint density formula, our Theorem \ref{thm:kappai-probabilistic-intro} (a polynomial upper bound with high probability) is rather coarse, but our theorem holds for general continuous matrices.  Besides our result, we are not aware of any results in the literature regarding diagonal overlaps for nonreal eigenvalues of the real Ginibre ensemble, or any other non-Hermitian random matrix model with real entries.

\paragraph{Eigenvector Condition Numbers and Numerical Analysis.} In 2007, Davies proposed a method for accurately computing analytic functions of matrices, $f(A)$, on machines with finite-precision arithmetic.  His insight was that adding a small independent complex Gaussian to each entry of a matrix $A$ improves the conditioning of its eigenvectors significantly, so that the approach of computing $f(A) \approx V\,f(D)\,V^{-1}$ becomes numerically stable. The quantitative relationship between the size of the Gaussian perturbation and resulting $\kappa_V$ formed the core conjecture of Davies' paper, and was confirmed by some of the authors and Mukherjee in the following theorems.

\begin{theorem}[{\cite[Theorem 1.5]{banks2019gaussian}}] \label{thm:davies} 
	Suppose $A\in\C^{n\times n}$ with $\|A\|\le 1$ and $\delta\in (0,1)$. Let $\bG_n$ be a complex Ginibre matrix, and let $\lambda_1,\ldots,\lambda_n\in \C$ be the (random) eigenvalues of $A+\delta \bG_n$.  Let $\vol_\C$ denote the Lebesgue measure on $\mathbb{C}$, then for every measurable open set $B\subset \C,$
	$$ 
		\dE \sum_{\lambda_i \in B} \kappa(\lambda_i)^2 \le \frac{n^2}{\pi \delta^2}\vol_\C(B).
	$$
\end{theorem}

\begin{theorem}[{\cite[Theorem 1.1]{banks2019gaussian}}]\label{thm:a} 
	Suppose $A\in\C^{n\times n}$ and $\delta\in (0,1)$. Then there is a matrix $E\in\C^{n\times n}$ such that $\|E\|\le \delta\|A\|$ and
	$$ 
		\kappa_V(A+E)\le 4n^{3/2}\left(1+\frac1\delta\right).
	$$
\end{theorem}

Notably, the theorems above do not address whether a real matrix can be regularized by a  \emph{real} perturbation, and \cite{banks2019gaussian} left this as an open question, which we resolve, albeit with a worse dependence on $\delta$.

\begin{remark}[Consequences for Diagonalization Algorithms] The nearly matrix multiplication time diagonalization algorithm of \cite{banks2019pseudospectral} uses perturbation of the input by a random complex Ginibre matrix as a crucial preprocessing step, guaranteeing that the perturbed matrix has $\kappa_V$ at most $\poly(n,1/\gamma)$ and eigenvalue gap at least $\poly(1/n,\gamma)$. Our Theorems \ref{thm:kappai-probabilistic-intro} and \ref{thm:gaps} imply that this continues to hold for perturbations satisfying Assumption 1, with slightly worse polynomial factors. This changes the running time of the algorithm by at most constant factors since that running time depends only logarithmically on these parameters.
\end{remark}

\paragraph{Singular Values of Real Matrices with Complex Shifts.}
In the course of our proof, it will be of particular importance to quantify the behavior of the small singular values of $z - \bG_n$ or $z - \bM_n$ as a function of the imaginary part of the complex scalar $z \in \C$.  There have already been a number of recent results in this direction, which we summarize below.

In the thesis of Ge \cite{ge2017eigenvalue} it was shown that when $\bM_n$ is a real matrix with i.i.d. entries of mean zero and variance $1/n$ satisfying a standard anticoncentration condition, one has 
\begin{equation} \label{eqn:gesingular}
    \P\left[\sigma_n(\bM_n-z)\leq \eps \text{ and } \|\bM_n\|\leq M \right]\leq \frac{Cn^2 \eps^2}{\im(z)}+e^{-cn}
\end{equation}
for all $z$, where $C$ and $c$ are universal constants, independent of $n$. The additional exponential term is an essential feature of the proof technique of considering ``compressible'' and ``incompressible'' vectors in a net argument, and does not go away if one additionally assumes that the entries are absolutely continuous.

In the case of real Ginibre matrices, the following finer result was obtained by Cipolloni, Erd\H{o}s and Schr\"{o}der in \cite{cipolloni2019optimal}:
\begin{equation} \label{eqn:cipollonisingular}
    \P\left[\sigma_n(\bG_n-z) \leq \eps \right]\leq
    C (n^2 (1 + |\log \eps |)\eps^2 + n \eps e^{-\frac{1}{2} n (\Im z )^2})
\end{equation}
for $|z| \le 1 + O(1/\sqrt{n})$, with an improved $n$-dependence at the edge $|z - 1| = O(1/\sqrt{n})$.  In later work \cite{cipolloni2020fluctuation}, the same authors showed that when $\bM_n$ has real i.i.d. entries with unit variance and $|\Im z| \sim 1$, the statistics of the small singular values $z - \bM_n$ agree with those of the complex Ginibre ensemble.\footnote{They further write,
    ``It is expected that the same result holds for all (possibly $n$-dependent) $z$ as long as $|\im(z)| \gg n^{-1/2}$, while in the opposite regime $|\im(z)| \ll n^{-1/2}$ the local statistics of the real Ginibre prevails with an interpolating family of new statistics which emerges for $|\im(z)| \sim n^{-1/2}$.''}

As remarked in the introduction, the key feature of our bounds is that we obtain a strict $\eps^2$ dependence for nonreal $z$, without any additive terms. Our approach is essentially different from the above two approaches, and relies on exploiting a certain conditional independence (Observation \ref{obs:mainobs}) between submatrices of the real and imaginary parts of the resolvent.

\paragraph{Singular Values of Real Matrices with Real Shifts.}
In the more general non-Gaussian case, there are a number of recent results in the literature.  The most relevant recent result is that of Nguyen \cite{nguyen2018random}, who proves a tail bound for all singular values for non-centered ensembles with potentially discrete entries. In the particular case of continuous entries, Nguyen shows that if $\bM_n$ satisfies Assumption \ref{assumption} with parameter $K>0$, 
\begin{equation} \label{eqn:nguyensingular}
    \mathbf{P}\left[\sigma_{n - k + 1}({\bM}_n) \le {\eps}\right] \le n^{k(k-1)} (C k K \eps)^{(k-1)^2},
\end{equation}
in addition to a bound greatly improving the dependence in $k$ at the expense of the dependence on $\eps$ and $n$, as well as results for symmetric Wigner matrices and perturbations thereof.  

The exponent of $\eps$ in \eqref{eqn:nguyensingular} is suboptimal, which \eqref{eqn:nguyensingular} incompatible with our approach. In Theorem \ref{thm:skboundgeneral} we obtain the optimal exponent of $\eps$, namely $k^2$, in exchange for a worse exponent of $n$.  The key ingredient in doing this is a simple ``restricted invertibility'' type estimate (Lemma \ref{lem:linearalgebralemma}) tailored to our setting.

For bounds on the least singular value alone, there is a substantial literature; see Table \ref{tab:centeredresults} for a non-exhaustive summary.

\begin{table}
\centering
\begin{threeparttable}[] 

\begin{tabular}{@{}lllll@{}}
\toprule
Result    & Bound & Setting \\ \midrule

\cite{edelman1988eigenvalues} & $\mathbb{P}[\sigma_n(\bM_n) < \eps] \le n\eps$ & real Ginibre\\

\cite{rudelson2008littlewood} & $\mathbb{P}[\sigma_n(\bM_n) < \eps] \le C n \eps + e^{-cn}$ & real i.i.d. subgaussian \\


\cite{tao2010randomsingular} & $\mathbb{P}[\sigma_n(\bM_n) < \eps] \le n\eps + O(n^{-c})$ & real i.i.d., finite moment assumption\\

\cite{sankar2006smoothed} & $\mathbb{P}[\sigma_n(A + \bM_n) < \eps] \le C n \eps$ & real Ginibre, $A$ real \\

\cite{tikhomirov2017invertibility} & $\mathbb{P}[\sigma_n(A + \bM_n) < \eps] \le C n \eps$ & real ind. rows with log-concave law, $A$ real\\

\cite{banks2019gaussian} & $\mathbb{P}[\sigma_n(A + \bM_n) < \eps] \le n \eps$ & real Ginibre, $A$ real \\


\bottomrule
\end{tabular}

\caption{Some bounds on $\sigma_n$ for real $\bM_n$ and $A$.  Entries of $\bM_n$ have variance $1/n$.} \label{tab:centeredresults}
\end{threeparttable}
\end{table}

\paragraph{Minimum Eigenvalue Gap.}
Bounds on the minimum eigenvalue gap of random non-Hermitian matrices have seen rapid progress in the last few years.  Ge shows in the thesis \cite{ge2017eigenvalue} that when $\bM_n$ has i.i.d. entries with zero mean and variance $1/n$, satisfying a standard anticoncentration condition,
\[
    \P[\gap(\bM_n) < s] = O\left( \delta n^{2+o(1)} + \frac{s^2 n^{4+o(1)}}{\delta^2}\right) + e^{-cn} + \P[\Vert \bM_n \Vert \ge M]
\]
for every $C > 0$ and every $\delta > s > n^{-C}$. In very recent work, Luh and O'Rourke \cite{luh2020eigenvectors} build on Ge's result, dropping the mean zero assumption and extending the range of $s$ all the way down to 0:
\begin{equation} \label{eqn:luhgap}
    \P[\gap(\bM_n) \le s \text{ and } \Vert \bM_n \Vert \le M] \le C s^{2/3} n^{16/15} + Ce^{-cn} + \P[\Vert \bM_n\Vert \ge M].
\end{equation}
However, (\ref{eqn:luhgap}) still requires the entries of $\bM_n$ to be identically distributed, so it does not imply a gap bound for the noncentered Ginibre ensemble $A + \bG_n$ unless $A$ is a scalar multiple of the all-ones matrix.

In our prior work \cite{banks2019pseudospectral}, a complex Gaussian perturbation was crucially used in a preprocessing step in a numerically stable diagonalization algorithm for non-Hermitian matrices. This paper identified the minimum eigenvalue gap as a key feature controlling the stability of the algorithm, and proved:

\begin{theorem}[{\cite[Corollary 3.7]{banks2019pseudospectral}}]
    \label{thm:smoothed} 
    Suppose $A\in \C^{n\times n}$ with $\|A\|\le 1$, and $\bG_n$ is a normalized \emph{complex} Ginibre matrix. For every $\delta \in (0,1/2)$, 
    $$
        \P[\gap(A + \delta \bM_n)<s]\le 42(n/\gamma)^{16/5}s^{6/5}+ \expbound. 
    $$
\end{theorem}
Each of the gap results above are proved by way of tail bounds on the smallest two singular values of $z - \bM_n$. The only other work we are aware of proving gap bounds for the case of matrices with i.i.d. entries is \cite{shi2012smallest}, which proves an inverse polynomial lower bound for the complex Ginibre ensemble.

\subsection{Probabilistic Preliminaries}
Many of our probabilistic arguments hinge on the phenomenon of \textit{anticoncentration}, whereby a random vector is unlikely to lie in a small region. An elementary way to extract quantitative information about such behavior is by controlling the density function of the random vector. Let $\bx \in \R^d$ be a random vector---we will always use boldface font to denote random variables. If the distribution $f_{\bx}$ of $\bx$ is absolutely continuous with respect to the Lebesgue measure on $\R^d$, we denote by
\begin{equation}
    \density(\bx) \defeq \|f_{\bx}\|_\infty
\end{equation}
the infinity norm of its density. We will use, ad nauseam, two basic observations about the quantity $\density$. First, for any $v \in \R^d$,
\begin{equation}
    \P\left[\|\bx - v\| \le \eps \right] \le \frac{\pi^{d/2}}{\Gamma(d/2 + 1)}\density(\bx)^d \le \frac{1}{\sqrt{\pi d}}\left(\frac{2e\pi}{d}\right)^{d/2}\density(\bx)^d,
\end{equation}
where in the first inequality we use the formula for the volume of a ball in $\R^d$, and in the second inequality we use Stirling's approximation for the gamma function. Second, $\density$ is preserved under convolution:



\begin{observation}[Convolution Bound]
    \label{obs:convbound}
    Let $\bx, \by \in \mathbb{R}^d$ be independent random vectors. Then 
    $$
        \density(\bx+\by)\leq \min \{ \density(\bx), \density(\by)\}.
    $$
\end{observation}

We will require as well a much more general result of Rudelson and Vershynin quantifying the deterioration of $\density$ after orthogonal projection\footnote{Throughout the paper, we will refer to a rectangular matrix with orthonormal columns as an ``orthogonal projection'' although this is not standard.}. 

\begin{theorem}[\cite{rudelson2015small}] \label{thm:rvdensity} \label{thm:rvproj}
Let $\bx \in \mathbb{R}^d$ have independent entries, each with density pointwise bounded by $K$.  Let $P \in \mathbb{R}^{k \times d}$ denote a deterministic orthogonal projection onto a subspace of dimension $k \le d$.  Then there exists a universal constant $\crv > 0$ such that  
\[ 
    \density(P\bx) \le (\crv K)^k.
\]
If $\bx$ has independent $\calN(0,1)$ entries, one may take $\crv = 1$ and $K = (2\pi)^{-1/2}$.
\end{theorem}

Many of our results on real random matrices whose independent entries have bounded density---in other words, matrices satisfying Assumption \ref{assumption}---can be strengthened for real Ginibre matrices. We have found the following comparison theorem to be a crucial tool in the Gaussian case:

\begin{theorem}[Real \'Sniady theorem]
    \label{thm:sniadyreal}
	Let $k \le n$, and let $A_1$ and $A_2$ be $n \times k$ real matrices, each with $k$ distinct singular values, such that $\sigma_i(A_1) \le \sigma_i(A_2)$ for all $i \in [k]$. Then for every $t \ge 0$, there exists a joint distribution on pairs of real  $n \times n$ random matrices $(\bX_1, \bX_2)$ such that 
	\begin{enumerate}
    	\item Each marginal $\bX_1$ and $\bX_2$ has independent $\calN(0,1)$ entries, and 
    	\item Almost surely $\sigma_i(A_1 + t \bX_1) \le \sigma_i(A_2 + t \bX_2)$ for all $1 \le i \le k$.
 	\end{enumerate} 
\end{theorem}

\noindent Theorem \ref{thm:sniadyreal} was originally discovered and proved for (square) complex Ginibre matrices in \cite{sniady2002random}.  The necessary technical modifications for the real case were carried out in \cite{banks2019gaussian}, and the proof there trivially extends to rectangular matrices. The virtue of Theorem \ref{thm:sniadyreal} is that one immediately obtains a remarkable stochastic dominance result relating the singular value distributions of non-centered Gaussian matrices.

\begin{corollary}
    \label{cor:sniady-dominance}
    Let $k \le n$, and let $A_1$ and $A_2$ be $n\times k$ real matrices satisfying $\sigma_i(A_1) \le \sigma_i(A_2)$ for all $i \in [k]$, and let $\bX$ have independent $\calN(0,1)$ entries Then, for any  $t,s_1,...,s_k \in \R_{\ge 0}$, 
    $$
        \P\big[\sigma_i(A_1 + t \bX) \le s_i\text{, } \forall i\in[k]\big] \ge \P\big[\sigma_i(A_2 + t \bX) \le s_i\text{, } \forall i \in [k] \big].
    $$
\end{corollary}

\noindent Although Theorem \ref{thm:sniadyreal} currently includes the technical assumption that $A_1$ and $A_2$ each have distinct singular values, Corollary \ref{cor:sniady-dominance} need not, by continuity of $\sigma_i(\cdot)$. We will most often apply Corollary \ref{cor:sniady-dominance} in the case when $A_1 = 0$, to transfer well-known singular value tail bounds from centered case to the non-centered one. For square Gaussian---that is, Ginibre---matrices, such tail bounds were proved by Szarek in \cite{szarek1991condition}.

\begin{theorem}[Szarek] \label{thm:szarek}
    Let $\bG_n$ be a normalized real Ginibre matrix. There exists a universal constant $c>0$ so that
    $$
        (c\eps)^{k^2} \le \P\left[\sigma_{n-k+1}(\bG_n) \le \frac{k\eps}{n} \right] \le (\sqrt{2e}\eps)^{k^2}.
    $$
    In the case of normalized complex Ginibre matrices, these bounds hold if one exchanges the exponent $k^2$ for $2k^2$.
\end{theorem}

\noindent Finally, one can bound the quantities $B_{\bG_n,p} = \E\left[\|\bG_n\|^p\right]^{1/p}$ explicitly in the Gaussian case: 
\begin{lemma} 
\label{lem:CGnp}
    Let $\bG_n$ be an $n\times n$ real Ginibre matrix and assume that $1\leq p\leq 2n$. Then $B_{\bG_n,p} \leq 9 $.
\end{lemma}
The proof proceeds by integrating well-known tail bounds on the operator norm of a Ginibre matrix, and is deferred to Appendix \ref{sec:ginibretail}.

\section{Anticoncentration} \label{sec:anticoncentration}

In this section we study the anticoncentration properties of certain quadratic functions of rectangular matrices with independent entries. These will be necessary in Section 3 to extract singular value tail bounds.

\begin{theorem}[Density of Quadratic Forms]
    \label{thm:anticoncentration}
    Assume that $\bX,\bY \in \R^{n\times k}$ are random matrices with independent entries, each with density on $\R$ bounded a.e. by $K>0$. Let $Z \in \R^{n\times n}$, $U,V \in \R^{n\times k}$, and $W \in \R^{k\times k}$ be deterministic, and write $q(\bX,\bY) \defeq \bX^\T Z \bY + \bX^\T U + V^\T \bY + W$. Then
    $$
        \density\left(q(\bX,\bY)\right) \le (1 + k^2)\left(\crv^2K^2 \sqrt{2e\pi k} \min_{j > k^2 + k + 1} \frac{1}{\sqrt{j-k+1} \sigma_j(Z)}\right)^{k^2}.
    $$
\end{theorem}

Whenever $\sigma_j(Z)$ is zero, we interpret $1/\sigma_j(Z) = \infty$; thus the above theorem has content only when $\rank (Z) > k^2 + k + 1$. After presenting the proof, we will comment on some improvements when $X,Y$ are Gaussian or $k = 1$. Let us begin with a small observation that we will use in the proof to come.

\begin{lemma} \label{lem:density-conditioning}
   Consider measurable functions $f: \mathbb{R}^p\times \mathbb{R}^q \to \mathbb{R}^r$ and $c: \mathbb{R}^q \to \mathbb{R}_{\geq 0}$. Let $\bx\in \mathbb{R}^p$ and $\by \in \mathbb{R}^q$ be independent random vectors with densities bounded almost everywhere. Assume that for almost all $y\in \mathbb{R}^{r}$ it holds that $\density\left(f(\bx, y)\right)\leq c(y).$ Then
   $$
    \density\left(f(\bx, \by)\right)\leq \E [c(\by)].
    $$
\end{lemma}

\begin{proof}
    Let $\leb_{\mathbb{R}^r}$  denote the Lebesgue measure on $\mathbb{R}^r$. Note that it is enough to show that for every measurable set $E \subset \mathbb{R}^r$ one has  
    $$
        \P[f(\bx, \by)\in E] \leq \leb_{\R^r}(E) \E[c(\by)].
    $$
    On the other hand, by assumption, we have $ \P[f(\bx, y) \in E] \le \leb_{\R^r}(E) c(y)$ for all $y$. From the fact that  $\bx$ and $\by$ are independent and have a density it follows that
    $$
        \P[f(\bx, \by)\in E] = \E[\indicator{f(\bx, \by)\in E}] = \E\left[\E\left[\indicator{f(\bx, \by)\in E}| \by \right] \right] \leq  \E\left[\leb_{\R^r}(E)c(\by)\right],$$
    as we wanted to show. 
\end{proof}

Second, we will require the following left tail bound on the smallest singular value of certain rectangular random matrices, which is a direct consequence of Theorem \ref{thm:rvdensity}.

\begin{lemma}
\label{lem: rectsingvals}
    Let $\bY$ be a $n\times k$ random matrix whose entries are independent and have density on $\R$ bounded a.e. by $K>0$. Furthermore, for some $k \le j \le n$ let $V$ be a $j\times n$ projector. Then
    \begin{equation}
        \label{eq:rectsingvals}
        \P[\sigma_k(V\bY)\leq s] \leq k \frac{(\crv K \sqrt{\pi k} s)^{j-k+1}}{\Gamma((j-k+3)/2)} \defeq C_{j,k} s^{j-k+1}
    \end{equation}
\end{lemma}

\begin{proof}
    Let $\by_1, \dots, \by_k$ be the columns of $\bY$ and for every $i=1, \dots, k$ let $\bW_i$ be the $(j - k + 1)\times j$ orthogonal projector onto the subspace orthogonal to the span of $\{V\by_l\}_{l\neq i}$. Applying the ``negative second moment identity'' \cite{tao2010random}, we have 
    $$
        k\left(\min_{i \in [k]} \|\bW_i V \by_i\|\right)^{-2} \ge \sum_{i=1}^k \|\bW_i V\by_i\|^{-2} \ge \sum_{i=1}^k \sigma_i(V\bY)^{-2} \ge k\sigma_k(V\bY)^{-2},
    $$
    which implies
    $$
        \sigma_k(\bY) \ge \frac{\min_i \|\bW_iV\by_i\|}{\sqrt k}.
    $$
    Since $\bW_iV$ is itself an orthogonal projector, and is independent of $\by_i$, Theorem \ref{thm:rvdensity} and Observation \ref{lem:density-conditioning} ensure that the density of $\|\bW_i V \by_i\|$ is bounded by $(\crv K)^{j-k+1}$. Applying a union bound and recalling again the formula for a ball,
    $$
        \P[\sigma_k(\bY) \le s] \le \P[\min_i \|\bW_iV\by_i\| \le \sqrt k s] \le \sum_{i=1}^k \P[\|\bW_iV\by_i\| \le \sqrt k s] \le k \frac{(\crv K \sqrt{\pi k} s)^{j-k+1}}{\Gamma((j-k+3)/2)}.
    $$
\end{proof}

\noindent With these two tools in hand, we proceed with the proof.

\begin{proof}[Proof of Theorem \ref{thm:anticoncentration}]
    For any deterministic $Y\in \mathbb{R}^{n\times k}$ one has $\density(q(\bX,Y)) = \density(\bX^T(ZY + U))$, since $\density$ is agnostic to deterministic translations. By the polar decomposition we can write $ZY + U = VS$, where $V \in \R^{n \times k}$ is a partial isometry and $S \succeq 0$. By Theorem \ref{thm:rvdensity}, the density of the random matrix $\bX^\T V$ in $\R^{k \times k}$ is at most $(\crv K)^{k^2}$, and thus the density of $\bX^\T V S$ is at most $(\crv K)^{k^2} (\det S)^{-k}$; moreover
    $$
        \det S = \prod_{i=1}^k\sigma_i(S) = \prod_{i=1}^k\sigma_i(ZY + U).
    $$
    Therefore by Lemma \ref{lem:density-conditioning},
    \begin{equation}
        \label{eq:qdensity}
        \density(q(\bX,\bY)) \le (\crv K)^{k^2} \E\left[\prod_{i\in k}\sigma_i(Z\bY + U)^{-k}\right].
    \end{equation}
    We now compute this expectation.
    
    Choose $j\ge k$ so that $\sigma_j(Z) > 0$, and write the SVD of $Z$ in the following block form,
    \begin{equation}
        \label{eq:zblock}
        Z = P^T\Sigma Q = \begin{pmatrix} P^\T_1 & P^\T_2 \end{pmatrix}\begin{pmatrix} \Sigma_1 & \\ & \Sigma_2 \end{pmatrix} \begin{pmatrix} Q_1 \\ Q_2 \end{pmatrix},
    \end{equation}
    where $\Sigma_1$ is a diagonal matrix containing the largest $j$ singular values, and $P,Q$ are orthogonal matrices. This gives
    $$
        Z\bY + U =\begin{pmatrix}P_1^\T & P_2^\T \end{pmatrix}
        \begin{pmatrix} \Sigma_1 Q_1 \bY + P_1 U \\ \Sigma_2 Q_2 \bY + P_2 U \end{pmatrix}.
    $$
    By interlacing of singular values, $\sigma_i(Z\bY + U) \ge \sigma_i(\Sigma_1 Q_1 \bY + P_1 U)$ for each $i = 1,...,k$, so we are free to study
    \begin{equation}
        \label{eq:sv-prod}
        \E\left[\prod_{i \in [k]}\sigma_i(\Sigma_1 Q_1 \bY + P_1 U)^{-k}\right] \le \sigma_j(\Sigma_1)^{-k^2} \E\left[ \prod_{i \in [k]}\sigma_i(Q_1 \bY + \Sigma_1^{-1}P_1 U)^{-k}\right].
    \end{equation}
    
    Now, since $Q_1$ is a partial isometry, we can select a matrix $\tilde U$ so that $Q_1 \tilde U = \Sigma_1^{-1}P_1 U$, and observe that
    $$
        \E\prod_{i \in [k]}\sigma_i(\Sigma_1 Q_1 \bY + P_1 U)^{-k} \le \sigma_j(Z)^{-k^2} \sigma_k(Q_1(\bY + \tilde U))^{-k^2}.
    $$
    The random matrix $\bY + \tilde U$ satisfies the conditions of Lemma \ref{lem: rectsingvals}, so we can apply the tail formula for expectation to obtain
    \begin{align*}
        \E\left[\sigma_k(Q_1(\bY + \tilde U))^{-k^2}\right]
        &= \int_0^\infty \P\left[\sigma_k(Q_1(\bY + \tilde U))^{-k^2} \ge t\right]dt \\
        &\le \lambda + C_{j,k}\int_\lambda^{\infty} t^{-\frac{j - k + 1}{k^2}}dt & & \text{$C_{j,k}$ from \eqref{eq:rectsingvals}}\\
        &= \lambda + C_{j,k}\frac{k^2}{j - k^2 - k + 1}\lambda^{\frac{k^2+k-j-1}{k^2}} & & \text{if } j - k + 1 > k^2. \\
    \end{align*} Optimizing the above bound in $\lambda$, we set $\lambda = C_{j,k}^{\frac{k^2}{j-k+1}}$ and evaluate $C_{j,k}$ to find
    \begin{align*}
        \E\left[\sigma_k(Q_1(\bY + \tilde U))^{-k^2}\right] 
        &\le \left(\frac{k (\crv K \sqrt{\pi k})^{j-k+1}}{\Gamma((j-k+3)/2)}\right)^{\frac{k^2}{j-k+1}}\left(1 + \frac{k^2}{j - k^2 - k + 1}\right) \\
        &\le (\crv K\sqrt{\pi k})^{k^2} \left(\frac{k}{\Gamma((j - k + 3)/2)}\right)^{\frac{k^2}{j - k + 1}}(1 + k^2) & & j-k+1 > k^2\\
        &\le (\crv K \sqrt{\pi k})^{k^2} \left(\frac{k}{\sqrt{\pi(j - k + 1)}}\right)^{\frac{k^2}{j-k+1}}\left(\frac{\sqrt{2e}}{\sqrt{j-k+1}}\right)^{k^2}(1 + k^2) & &  \text{Stirling}\\
        &\le \left(\frac{\crv K  \sqrt{2e\pi k}}{\sqrt{j-k+1}}\right)^{k^2} (1+k^2) & & j-k+1 > k^2
    \end{align*}
    where we have repeatedly used that $j-k+1 > k^2$, as well as Stirling's approximation, $\Gamma(z + 1) \ge \sqrt{2\pi z}(z/e)^z$, valid for real $z \ge 2$. To complete the proof, we combine the above with equation \eqref{eq:qdensity}.
\end{proof}

To end this section, we offer some improvements of the above for small $k$ or Gaussian $\bX$ and $\bY$.

\begin{corollary}
    In the case $k=1$, the conclusion of Theorem \ref{thm:anticoncentration} may be improved to
    $$
        \density\left(q(\bX,\bY)\right) \le 2 (\crv K)^2 \sqrt{2e\pi} \min_{j\ge 2}\frac{1}{\sqrt j \prod_{i \in [j]}\sigma_i(Z)^{1/j}.}
    $$
    Recall that in the Gaussian case, we may take $\crv = 1$ and $K = (2\pi)^{-1/2}$.
\end{corollary}
\begin{proof}
    The discussion between equations \eqref{eq:qdensity} and \eqref{eq:sv-prod} in this case tells us
    $$
        \density(\left(q(\bX,\bY)\right) \le \crv K \E\left[\|\Sigma_1Q_1\bY + P_1 U\|^{-1}\right].
    $$
    The random vector $\Sigma_1 Q_1 \bY + P_1 U$ has density on $\R^j$ bounded by $(\crv K)^j \det\Sigma_1^{-1}$, so we have the tail bound
    $$
        \P\left[\|\Sigma_1Q_1 \bY + P_1 U\| \le s\right] \le \det \Sigma_1^{-1} \frac{(\crv K \sqrt\pi s)^j}{\Gamma(j/2 + 1)} = \det \Sigma_1^{-1} \cdot C_{j,1} s^j. 
    $$
    Replacing in the remainder of the proof $C_{j,k}$ with $\det \Sigma_1^{-1}C_{j,1}$, and recalling $\det \Sigma_1 = \sigma_1(Z) \cdots \sigma_j(Z)$, will give
    $$
        \density\left(q(\bX,\bY)\right) \le \crv K \E\left[\|\Sigma_1 Q_1 Y + P_1 U\|^{-1}\right] \le 2\frac{(\crv K)^2\sqrt{2e\pi}}{\sqrt j \prod_{i \in [j]} \sigma_i(Z)^{1/j}}
    $$
    whenever $j\ge 2$.
\end{proof}

We believe that Theorem \ref{thm:anticoncentration} should hold, for every $k$, the $j$th singular value of $Z$ exchanged for the geometric mean of the top $j$. The main obstacle seems to be that the Theorem \ref{thm:rvdensity} cannot tightly bound the density of $A\by$, where $\by \in \R^{n}$ is a random vector with independent entries and bounded density, and $A \in \R^{n\times k}$ is an arbitrary matrix.

In a different direction, one can improve the constant in Theorem \ref{thm:anticoncentration} under a Gaussian assumption.

\begin{gtheorem}{\ref{thm:anticoncentration}}
    \label{thm:anticoncentration-gaussian}
    If $\bX,\bY \in \R^{n\times k}$ have independent, standard Gaussian entries, then t Theorem \ref{thm:anticoncentration} holds with the stronger conclusion:
    \begin{equation}
        \density\left(q(\bX,\bY)\right) \le \left(\frac{1}{2} \min_{j > 2k} \frac{1}{\sqrt{j-2k+1} \sigma_j(Z)}\right)^{k^2}.
    \end{equation}
\end{gtheorem}

\begin{proof}
    Once again we modify the proof beginning at \eqref{eq:sv-prod}. Observing that $Q_1 \bY + \Sigma_1^{-1}P_1U$ is a $j\times k$, non-centered Gaussian matrix, Theorem \ref{thm:sniadyreal} implies
    $$
        \E\prod_{i=1}^k\sigma_i(Q_1 \bY + \Sigma_1^{-1}P_1 U)^{-k^2} \le \E \prod_{i=1}^k\sigma_i(Q_1\bY)^{-k} = \E (\det \bY^\T Q_1^\T Q_1 Y)^{-k^2/2}.
    $$
    Now, $\bY^\T Q_1^\T Q_1 \bY$ is a real Wishart matrix with parameters $(j,k)$, and it is known \cite{goodman1963distribution} that the determinant of such a matrix is distributed as a product of independent $\chi^2$ random variables $\bnu_j \bnu_{j-1} \cdots \bnu_{j-k+1}$, where $\bnu_l \sim \chi^2(l)$. Computing directly,
    $$
        \E \bnu_l^{-k/2} = \int_{0}^\infty \frac{x^{l/2 - k/2-1}\exp(-x/2)}{2^{l/2}\Gamma(l/2)} = \frac{2^{-k/2} \Gamma \left((l-k)/2\right)}{\Gamma \left(l/2\right)},
    $$
    whenever $l > k$. For even $k$, this has the closed form $(l-2)^{-1}(l-4)^{-1}\cdots(l-k)^{-1} \le (l-k)^{-k/2}$. This final bound holds for odd $k\ge 3$, by repeated application of $z\Gamma(z) = \Gamma(z + 1)$ and one use of the inequality $\sqrt{2z/\pi}\Gamma(z) \le \Gamma(1/2 + z) \le \sqrt z \Gamma(z)$, valid for all $z\ge 1/2$. When $k=1$, this inequality again gives us $\E\nu_l^{-1/2} \le (\pi(l-1)/2)^{-1/2}$. As above, we can take $\crv = 1$ and $K = (2\pi)^{-1/2}$ in the Gaussian case, so
    \begin{align*}
        \density\left(q(\bX,\bY)\right) 
        &\le \left(\frac{1}{\sqrt{2\pi}\sigma_j(Z)}\right)^{k^2}\prod_{l = j-k+1}^k \E \nu_l^{-k/2} \\
        &\le \left(\frac{1}{2 \sigma_j(Z)}\right)^{k^2}\prod_{l = j-k+1}^j (l - k)^{-k/2} \\
        &\le \left(\frac{1}{2\sqrt{j-2k+1} \sigma_j(Z)}\right)^{k^2}.
    \end{align*}
    The condition $j>2k$ ensures that each $\E \nu_l^{-k/2} < \infty$ for $l = j-2k+1,...,j$.
\end{proof}

\section{Singular Value Bounds for Non-Centered Real Matrices}\label{sec:realreal}
In this section, we discuss singular value tail bounds for real matrices with independent absolutely continuous entries. In particular, our study of minimum eigenvalue gap and eigenvalue condition numbers will require tail bounds on the least two singular values for shifted random matrices of the form $z - (A + \bM_n)$, where $z \in \R$ and $A \in \R^{n\times n}$ are deterministic, and $\bM_n$ satisfies Assumption~\ref{assumption}.

As a warm-up, we obtain as an immediate consequence of Theorem \ref{thm:szarek} and Corollary \ref{cor:sniady-dominance}---Szarek's singular value bounds for centered real Ginibre matrices, and the stochastic dominance corollary to \'Sniady's Comparison Theorem---that
\begin{equation}
    \label{eq:ginibre-real-sv}
    \P\left[\sigma_{n-k+1}(z - (A + \gamma \bG_n)) \le \eps \right] \le \left(\frac{\sqrt{2e}n\eps}{k\gamma}\right)^{k^2}
\end{equation}
for every $\gamma > 0$ and $k \in [n]$. This $\eps^{k^2}$ behavior will be a useful benchmark by which to assess our results below.

For matrices with i.i.d. subgaussian entries, results similar to Szarek's theorem are known, but they are accompanied by additive error terms of the form $e^{-cn}$ and therefore do not yield useful results in the limit as $\eps \to 0$. The closest result to ours appears in \cite{nguyen2018random}; it excises the additive error terms, but contains a sub-optimal exponent on $\eps$. We will add one key insight to Nguyen's proof that allows one to obtain the correct $\eps$-dependence.

\subsection{A Restricted Invertibility Lemma}
The device we add to Nguyen's argument, and which we will return to at several points throughout the paper, is the following lemma, which shows that the $k$th largest eigenvalue of a PSD matrix is approximately witnessed by the smallest eigenvalue of some principal $k\times k$ submatrix.
\begin{lemma}[Principal Submatrix with Large $\sigma_k$] \label{lem:linearalgebralemma}
    Let $X \in \mathbb{C}^{n \times n}\setminus\{0\}$ be positive semidefinite. Then for every $1\le k\le n$, there exists an $k\times k$ principal submatrix $X_{S,S}$ such that
    \begin{equation}\label{eqn:linalg}
    \lambda_k(X_{S,S}) \ge \frac{\Tr(X)}{\sum_{i=1}^k \lambda_i(X)}\cdot \frac{\lambda_k(X)}{k(n-k+1)}.
    \end{equation}
  \end{lemma}

\begin{proof}
Examining the coefficient of $\lambda^k$ in the characteristic polynomial $\det(\lambda - X)$, we have
\[ \sum_{|S| = k} \det X_{S,S} = e_k(\lambda_1(X), \lambda_2(X), \dots, \lambda_n(X)),\]
where $e_k$ denotes the $k$-th elementary symmetric function, and the sum runs over subsets of $[n]$. We may now  have the upper bound:
\begin{align*}
    e_k(X) &= \sum_{|S|=k} \det(X_{S,S})\\
            &= \sum_{|S|=k} \lambda_k(X_{S,S})\lambda_{k-1}(X_{S,S})\ldots \lambda_1(X_{S,S})\\
            &\le \sum_{|S|=k} \lambda_k(X_{S,S})e_{k-1}(X_{S,S})& &\textrm{since $\lambda_i(X_{S,S})\ge 0$ by interlacing}\\
            &\le \max_S \lambda_k(X_{S,S})\cdot \sum_{|S|=k}\sum_{T\subset S, |T|=k-1} \det(X_{S',S'})\\
            &=\max_S \lambda_k(X_{S,S})\cdot (n-k+1)e_{k-1}(X).
\end{align*}        
It now remains to furnish a complementary lower bound on $e_k(X)$ in terms of $e_{k-1}(X)$. Recall the routine fact that
\begin{align*}
    k e_k(X) = k \sum_{|S| = k} \prod_{i \in S}\lambda_i(X) = \sum_{|T| = k-1} \sum_{j \notin T} \lambda_j(X) \prod_{i \in T} \lambda_i(X).
\end{align*}
Now, for each $|T| = k-1$,
\begin{align*}
    \sum_{j \in [k]} \lambda_j(X) \sum_{\ell \notin T}\lambda_\ell(X)
    &=  \sum_{j \in [k]} \lambda_j(X) \left(e_1(X) - \sum_{j \in T}\lambda_j(X)\right) \\
    &= \lambda_k(X) e_1(X) +  \left(\sum_{j \in [k-1]} \lambda_j(X)\right)e_1(X) - \left(\sum_{j \in T}\lambda_j(X)\right)\left(\sum_{j \in [k]}\lambda_j(X)\right) \\
    &\ge \lambda_k(X) e_1(X),
\end{align*}
since $\sum_{j \in [k-1]} \lambda_j(X) \ge \sum_{j \in T} \lambda_j(X)$, and $e_1(X) \ge \sum_{j \in [k]} \lambda_j(X)$. Thus
$$
    k\sum_{j \in [k]}\lambda_j(X) \cdot e_k(X) \ge \sum_{|T|= k-1} \lambda_k(X)e_1(X) \prod_{i \in T}\lambda_i(X) = \lambda_k(X)e_1(X)e_{k-1}(X).
$$
 Putting everything together, and recalling $e_1(X) = \Tr X$, $$
    \max_S \lambda_k(X_{S,S}) \ge \frac{e_k(X)}{(n-k+1)e_{k-1}(X)} \ge \frac{\Tr(X)}{\sum_{i \in [k]} \lambda_i(X)} \frac{\lambda_k(X)}{k(n-k+1)}
$$
as desired.
\end{proof}

We will employ Lemma \ref{lem:linearalgebralemma} in the form of the corollary below.

\begin{corollary}
    \label{cor:linearalgnonpsd}
    Let $1 \le k \le n$. For every matrix $R \in \C^{n\times k}$, there exists a $k\times k$ submatrix $Q$ of $R$ such that
    $$
        \sigma_k(Q) \ge \frac{\sigma_k(R)}{\sqrt{k(n-k+1)}}.
    $$
    Similarly, for every matrix $A \in \C^{n\times n}$, there are subsets $S,T \subset [n]$ of size $k$ such that
    $$
        \sigma_k(A_{S,T}) \ge \frac{\|A\|_F}{\sqrt{\sum_{i \in [k]} \sigma_i(A)^2}} \frac{\sigma_k(A)}{k(n-k+1)} \ge \frac{\sigma_k(A)}{k(n-k+1)}
    $$
\end{corollary}
\noindent This generalizes the elementary fact that the operator norm of an $n\times n$ matrix is bounded above by $n$ times the maximal entry. Corollary \ref{cor:linearalgnonpsd} additionally sits within a much larger literature on \textit{restricted invertibility}; see \cite{naor2017restricted} for a comprehensive introduction.

\subsection{Proof of Theorem \ref{thm:skboundgeneral}}
Finally, we may prove the desired tail bound:

\begin{rtheorem}{\ref{thm:skboundgeneral}}
Let  $\bM_n \in \mathbb{R}^{n \times n}$ be a random matrix satisfying Assumption~\ref{assumption} with parameter $K>0$.  Then
$$\P\left[\sigma_{n-k+1}(\bM_n) \leq \eps \right]  \le \binom{n}{k} \left(\crv K \eps \sqrt{kn(n - k + 1)} \right)^{k^2} \le n^{k^2 + k} k^{\frac{1}{2} k^2}(\crv K)^{k^2} \eps^{k^2}.$$
\end{rtheorem}
\begin{proof}[Proof of Theorem \ref{thm:skboundgeneral}]
    We repeat the argument of Nguyen \cite{nguyen2018random}, but using Corollary \ref{cor:linearalgnonpsd} where Nguyen uses the restricted invertibility theorem of \cite{naor2017restricted}. 
    
    Suppose $\sigma_{n-k+1}(\bM_n) \le \eps$. By the minimax formula for singular values, there exist orthogonal unit vectors $\bz_1, \dots, \bz_k \in \mathbb{R}^n$ such that $\Vert \bM_n \bz_i \Vert \le \eps$.  Letting $\bZ \in \mathbb{R}^{n \times k}$ be the matrix whose columns are $\bz_1, \dots, \bz_k$, we can bound $\Vert \bM_n\bZ \Vert_F \le \eps \sqrt{k}$. Since $\sigma_k(\bZ) = 1$, by Corollary \ref{cor:linearalgnonpsd}, there is a $k\times k$ submatrix $\bZ_1$ of $\bZ$ for which
    $$
        \|\bZ_1\|^{-1} \le \sqrt{k(n-k+1)}.
    $$
    Denote by $\bZ$ the subset of rows of $\bZ$ participating in $\bZ_1$; by permuting if necessary we can write
    $$
        \bZ = \begin{pmatrix}\bZ_1 \\ \bZ_2\end{pmatrix} \qquad \text{and} \qquad \bM_n = \begin{pmatrix}\bM_1 & \bM_2 \end{pmatrix},
    $$
    observing that 
    \begin{equation}
        \label{mza}
        \bM\bZ\bZ_1^{-1} = \begin{pmatrix}\bM_1 & \bM_2\end{pmatrix} \begin{pmatrix}\bZ_1 \\ \bZ_2 \end{pmatrix}\bZ_1^{-1} = \bM_1 + \bM_2 \bZ_2\bZ_1^{-1}.
    \end{equation}
    
    Denote the columns of $\bM_n$ by $\bm_1, \dots, \bm_n$ and let $\bH$ denote the orthogonal projector onto the $k$-dimensional subspace orthogonal to the span of $\{\bm_i\}_{i \not\in \bS}$, so that $\bH \bM_2 = 0$. Thus we have
    \[ 
        \sum_{i \in \bS} \|\bH \bm_i\|^2 = \|\bH\bM\bZ\bZ_1^{-1}\|_F^2 
        \le \Vert \bM_n\bZ \bZ_1^{-1} \Vert_{F}^2 
        \le \Vert \bM_n\bZ \Vert_{F}^2 \Vert \bA^{-1} \Vert^2 
        \le \eps^2 k^2(n-k+1).
    \]
    Since the entries of $\bM_n$ are independent, with densities on $\R$ bounded by $\sqrt n K$, by Theorem \ref{thm:rvproj} the above event occurs with probability at most
    \[ 
        \prod_{i=1}^k\P\left[\|\bH \bm_i\| \le \eps k\sqrt{n-k+1}\right] < \left(\crv K\sqrt n \cdot \eps \sqrt{k(n-k+1)}\right)^{k^2}.
    \]
    Performing a union bound over all possibilities for the subset $\bS$ of rows of $\bZ$, we finally obtain
    \[ 
        \P\left[\sigma_{n-k+1}(\bM_n) \leq \eps \right]  \le \binom{n}{k} \left(\crv K \eps \sqrt{kn(n - k + 1)} \right)^{k^2} \le n^{k^2 + k} k^{\frac{1}{2} k^2}(\crv K)^{k^2} \eps^{k^2}.
    \]
\end{proof}

Comparing with Szarek's result (Theorem \ref{thm:szarek} above), we conclude that the exponent of $\eps$ in Theorem \ref{thm:skboundgeneral} is optimal, and if not for the factor of $\binom{n}{k}$ arising from the union bound, the exponent of $n$ would be optimal as well. Since we made no requirement that $\bM_n$ is centered, the following corollary is immediate:
\begin{corollary} \label{cor:realrealsingulargamma}
    Let $z \in \R$ and $A \in \mathbb{R}^{n \times n}$ be deterministic, and $\bM_n$ satisfy Assumption \ref{assumption} with parameter $K>0$. Then
    \[
        \P[\sigma_{n-k+1}(z - (A + \bM_n)) \le \eps] \le n^{\frac{1}{2}k^2 + k} k^{\frac{1}{2} k^2}(\crv K)^{k^2} \eps^{k^2}.
    \]
\end{corollary}

We record our initial observation regarding real Ginibre matrices, equation \eqref{eq:ginibre-real-sv}, as the following theorem.

\begin{rtheorem}{\ref{thm:gaussianrealrealsingular}}
    Let $z \in \R$ and $A \in \mathbb{R}^{n \times n}$ be deterministic, and $\bG_n$ be a normalized Ginibre matrix. For every $\gamma > 0$,
    \[
        \P[\sigma_{n-k+1}(z - (A + \gamma \bM_n)) \le \eps] \le \left(\frac{\sqrt{2e} n \eps}{k \gamma} \right)^{k^2}.
    \]
    In the case $k = 1$, one has a better constant:
    \[ \P[\sigma_{n}(z - (A + \gamma \bM_n)) \le \eps] \le \frac{n \eps}{\gamma}.\]
\end{rtheorem}
\begin{proof}
    When $A=0$, this is Theorem \ref{thm:szarek}, and the better constant for $k=1$ is a result of Edelman \cite{edelman1988eigenvalues}.  The conclusion for general $A$ then follows from Corollary \ref{cor:sniady-dominance}.
\end{proof}

\section{Singular Value Bounds for Real Matrices with Complex Shifts}
\label{sec:realcomplex}

In order to control the eigenvalue gaps and pseudospectrum of random real perturbations, we need to understand the smallest singular values of real random matrices with complex scalar shifts. As  discussed in the introduction, our results will be stated in terms of the quantities
$$
    B_{\bM_n,p} \defeq \left[ \E\|\bM_n\|^p\right]^{1/p},
$$
and important features of the bounds in our context are (1) the optimal dependence on $\eps$ as $\eps \to 0$, and (2) the factor $\frac{1}{|\im z|}$ controlling the necessary deterioration of the bound as $z$ approaches the real line.
\begin{rtheorem}{\ref{thm:singvalscomplexshifts-intro}}
\label{thm:singvalscomplexshifts}
    Let $z\in \C\setminus \R$ and $A\in\R^{n\times n}$ be deterministic, and let $\bM_n$ satisfy Assumption \ref{assumption} with parameter $K>0$. For every $k \le \sqrt n - 2$,
    $$
        \P\left[\sigma_{n-k+1}\left(z - (A + \bM_n)\right) \le \eps \right] \le (1 + k^2){\binom{n}{k}}^2\left(C_{\ref{thm:singvalscomplexshifts}} k^2(nK)^3\left(\left( B_{\bM_n,2k^2} + \|A\| + |\Re z|\right)^2 + |\Im z|^2 \right)\frac{\eps^2}{|\Im z|}\right)^{k^2},
    $$
    \color{darkgray}
    \color{black}
    where $C_{\ref{thm:singvalscomplexshifts-intro}}$ is a universal constant defined in \eqref{eq:singvalscomplexshifts-const}.
\end{rtheorem}

\noindent In the Gaussian case, we can excise this factor of $(1 + k^2)$ and extend the range of $k$.

\begin{rtheorem}{\ref{prop:gaussiansingvalscomplex-intro}}
\label{prop:gaussiansingvalscomplex}
    Let $z\in\C \setminus \R$ and $A \in \R^{n\times n}$ be deterministic, and let $\bG_n$ be a normalized $n\times n$ real Ginibre matrix. For every $\gamma > 0$, and every $k \le n/7$, 
    $$
        \P\left[\sigma_{n-k+1}(z - (A + \gamma \bG_n)) \le \eps \right] \le {\binom{n}{k}}^2 \left(\frac{\sqrt{7e} k^2 n^3}{2\gamma^3} \left(\left(\gamma B_{\bG_n,2k^2} + \|A\| + |\Re z|\right)^2 + |\Im z|^2\right)\frac{\eps^2}{|\Im z|}\right)^{k^2}.
    $$
\end{rtheorem}

\subsection{Proof of Theorem \ref{thm:singvalscomplexshifts-intro}}
In view of Corollary \ref{cor:linearalgnonpsd}, we can study the $k$th smallest singular value of $z - (A + \bM_n$) by examining the smallest singular value of every $k\times k$ submatrix of its inverse. In particular, we will show momentarily that Theorem \ref{thm:singvalscomplexshifts} may be reduced to the following lemma, which we will prove in Section \ref{subsec:lemmaproof}. Theorem \ref{prop:gaussiansingvalscomplex} requires only a few small modifications to the arguments of the general case, and we defer the proof until Section \ref{subsec:gaussproof}.

\begin{lemma}[Tail bound for corner of the resolvent]
\label{lem:taildiagonalentries}
    Let $\delta \in \R$, let $U$ be a permutation matrix, and let $\bM_n$ satisfy Assumption \ref{assumption} with parameter $K>0$. Denote the upper-left $k\times k$ corner of  $(\delta i U - \bM_n )^{-1}$ by $\bN_k$. If $n \ge (k+2)^2$,
    \begin{equation}
    \label{eq:tailbound11entry}
        \P\left[\sigma_k(\bN_k) \geq 1/\eps \right]\leq (1 + k^2)\left(\sqrt{6}\crv^2 (2e\pi)^{3/2} K^3 n \frac{\eps^2}{|\delta|}\right)^{k^2}\E\left[\left(\|\bM_n\|^2 + \delta^2\right)^{k^2}\right].
    \end{equation}
\end{lemma}

We now show that Lemma \ref{lem:taildiagonalentries} implies Theorem \ref{thm:singvalscomplexshifts}. The proof of Lemma \ref{lem:taildiagonalentries} is deferred to Section \ref{subsec:lemmaproof} and is the main technical work of the proof.

\begin{proof}[Proof of Theorem \ref{thm:singvalscomplexshifts} assuming Lemma \ref{lem:taildiagonalentries}]
    Applying Corollary \ref{cor:linearalgnonpsd} and a union bound,
    \begin{align}
        \label{eq:unionbound}
        \P\left[\sigma_{n-k+1}(z - (A + \bM_n))\leq \eps \right] 
        &= \P\left[\sigma_k\left((z - (A + \bM_n))^{-1}\right) \geq 1/\eps \right] \nonumber \\ 
        &\leq \P\left[ \max_{S, T\subset [n], |S|=|T|=k} \sigma_k\left((z - (A + \bM_n))_{S, T}^{-1}\right) \geq  \frac{1}{k(n-k+1)\eps}\right] \nonumber \\ 
        &\leq \sum_{S, T\subset [n], |S|=|T|=k} \P\left[\sigma_k\left((z - (A + \bM_n))_{S, T}^{-1}\right) \geq  \frac{1}{k(n-k+1)\eps} \right].
    \end{align}
    Fixing $S, T\subset[n]$ of size $k$, there are permutation matrices $P$ and $Q$ such that 
    \begin{align*}
        (z - (A + \bM_n))^{-1}_{S,T}
        &= \left(Q^\T(z - (A +  \bM_n))^{-1}P\right)_{[k], [k]} \\
        &= \left(PQ^\T i\im z + P(\Re z - (A +  \bM_n))Q^\T\right)^{-1}_{[k],[k]}.
    \end{align*}
    As $PQ^\T$ is a permutation matrix and $P(\Re z - (A +  \bM_n))Q^\T$ satisfies Assumption \ref{assumption} with parameter $K>0$, we can apply Lemma \ref{lem:taildiagonalentries}. Defining
    \begin{equation}
        \label{eq:singvalscomplexshifts-const}
        C_{\ref{thm:singvalscomplexshifts}} \defeq \sqrt{6}\crv^2 (2e\pi)^{3/2},
    \end{equation}
    this gives
    \begin{align*}
        & \P\left[ \sigma_k\left((z-(A+ \bM_n) )^{-1}_{S, T}\right)\geq \frac{1}{k(n-k+1)\eps } \right] 
        \\ = & \P\left[\sigma_k\left(i\Im z  PQ^\T - P(\Re z - (A +  \bM_n))Q^\T\right)^{-1}_{[k],[k]} \ge \frac{1}{k(n-k+1)\eps}\right] \\ 
        \le & (1 + k^2)\left(C_{\ref{thm:singvalscomplexshifts}} K^3n \frac{ k^2 (n-k+1)^2\eps^2}{|\Im z|}\right)^{k^2} \E \left[ \left(\|P(\Re z - A +  M_n)Q^\T\|^2 + |\Im z|^2\right)^{k^2} \right] \\
        \le & (1 + k^2)\left(C_{\ref{thm:singvalscomplexshifts}} k^2 n^3 K^3 \frac{\eps^2}{|\Im z|}\right)^{k^2} \E\left[ \left(\|P(\Re z - (A + \bM_n))Q^\T\|^2 + |\Im z|^2\right)^{k^2}\right],
    \end{align*}
    where we have bounded $n-k+1 \le n$. By Jensen, $B_{\bM,s} \le B_{\bM,t}$ for any random matrix $\bM$ and $s\le t$, and thus expanding out with the binomial theorem gives $B_{A+\bM, s} \le B_{\bM,s} + \|A\|$ for every deterministic $A$. Finally,
    \begin{align*}
        \E\left[\left(\|P(\Re z - (A + \bM_n))Q^\T\|^2 + |\Im z|^2\right)^{k^2}\right] 
        &=  \E \left[\left(\|\Re z - (A + \bM_n)\|^2 + |\Im z|^2\right)^{k^2}\right] \\
        &= \sum_{r = 0}^{k^2}{\binom{k^2}{r}} B_{\Re z - (A + \bM_n), 2r}^{2r}|\Im z|^{2k^2-2r} \\
        &\le (B_{\Re z - (A +  \bM_n),2k^2}^2 + |\Im z|^2)^{k^2} \\
        &\le \left(( B_{\bM_n,2k^2} + \|A\| + |\Re z|)^2 + |\Im z|^2\right)^{k^2}.
    \end{align*}
    
    \color{darkgray}
    
    \color{black}
    We finish by combining this with the previous equation, and multiplying by ${\binom{n}{k}}^2$ for the union bound over pairs of size-$k$ subsets $S$ and $T$.
\end{proof}
\subsection{Proof of Lemma \ref{lem:taildiagonalentries}}
\label{subsec:lemmaproof}

In what follows we use the notation and assumptions of Lemma \ref{lem:taildiagonalentries}. In particular, $\bM_n$ satisfies Assumption \ref{assumption} with parameter $ K > 0$, $U$ is a permutation matrix, and $\delta \in \R$. Once again writing $\bN_k$ for the upper left $k\times k$ block of $(\delta i U + \bM_n)^{-1}$, we need to show that $\P[\|\bN_k^{-1}\| \le \eps] = O(\eps^{2k^2})$. One would expect this behavior if the real and imaginary parts of $\bN_k^{-1}$ were independent, and each had a density on $\R^{k\times k}$. We will not be quite so lucky, but we \textit{will} be able to separate the randomness in its real and imaginary parts, obtaining the $O(\eps^{2k^2})$ behavior by conditioning on some well-chosen entries of $\bM_n$. To make this precise, we will need some notation.

Let us write $\bM_n$ and $\delta U$ in the following block form:
\begin{equation}
\label{eq:blockdecomposition}
    \bM_n = 
    \begin{pmatrix}
        \bM_{11} &\bM_{12} \\
        \bM_{21} &\bM_{22}
    \end{pmatrix} 
    \quad \text{and} \quad 
    \delta U=\begin{pmatrix}
        U_{11} &U_{12} \\
        U_{21} &U_{22}
    \end{pmatrix}
\end{equation}
where $\bM_{11}$ and $U_{11}$ are $k\times k$ matrices. Define as well the $(n-k)\times (n-k)$ matrices $\bX$ and $\bY$ as
\begin{equation}
    \label{eq:defofXY}
    \bX \defeq \Re (\bM_{22}+i  U_{22})^{-1} \quad \text{and} \quad  \bY \defeq \Im (\bM_{22}+i  U_{22})^{-1} .
\end{equation}
Applying the Schur complement formula to the block decomposition in \eqref{eq:blockdecomposition}, we get
\begin{align*}
    \bN_k^{-1} &= \bM_{11}+iU_{11}- (\bM_{12}+iU_{12})(\bM_{22}+i U_{22} )^{-1}(\bM_{21}+iU_{21})
    \\ &=\bM_{11}+iU_{11} - (\bM_{12}+iU_{12})(\bX+i\bY)(\bM_{21}+iU_{21}),
\end{align*}
meaning that
\begin{align}
    \Re \bN_k^{-1} &= \bM_{11} - \bM_{12}\bX \bM_{21} + U_{12}\bY \bM_{21} - \bM_{12}\bY U_{21} + U_{12}\bX U_{21} \label{eq:realpart}\\
    \Im \bN_k^{-1} &= U_{11} -\bM_{12} \bY \bM_{21} -\bM_{12} \bX U_{21} - U_{12} \bX \bM_{21} + U_{12} \bY U_{21}. \label{eq:imagpart}
\end{align}

Examining these two formulae, and recalling that the entries of $\bM_n$ are independent and have a joint density on $\R^{n\times n}$, we arrive at the key observation of this section:
\begin{observation}
    \label{obs:mainobs}
    The imaginary part $\Im \bN_k^{-1}$ is independent of $\bM_{11}$. Moreover, conditional on $\bM_{12},\bM_{21}$ and $\bM_{22}$, the real part $\Re \bN_k^{-1}$ has independent entries, each with density on $\R$ bounded by $K\sqrt n$.
\end{observation}

\noindent Writing this conditioning explicitly,
\begin{align}
    \P\left[\sigma_k(\bN_k) \ge 1/\eps \right] 
    &= \P\left[\|\bN_k^{-1}\| \le \eps \right] \nonumber \\
    &\le \P\left[\|\Re \bN_k^{-1} + i \Im \bN_k^{-1}\|_F \le \eps \sqrt k\right] \nonumber \\
    &\le \P\left[\|\Re \bN_k^{-1}\|_F \le \eps, \|\Im \bN_k^{-1}\|_F \le \eps\sqrt k \right] \nonumber \\
    &= \E \, \E\left[\indicator{\|\Re \bN_k^{-1} \|_F \le \eps\sqrt k} \indicator{\|\Im \bN_k^{-1}\|_F \le \eps\sqrt k} \,\middle|\, \bM_{12},\bM_{21},\bM_{22}\right] \nonumber \\
    &= \E\left[\indicator{\|\Im \bN_k^{-1}\|_F \le \eps\sqrt k} \E\left[\indicator{\|\Re \bN_k^{-1}\|_F \le \eps\sqrt k} \md \bM_{12},\bM_{21},\bM_{22}\right]\right]. \label{eq:nk-conditional}
\end{align}
We can bound the inner conditional expectation using Observation \ref{obs:mainobs}: 
\begin{align}
    \E\left[\indicator{\|\Re \bN_k^{-1}\|_F \le \eps\sqrt k} \md \bM_{12},\bM_{21},\bM_{22}\right]
    &\le \frac{(\sqrt{\pi kn}K\eps)^{k^2}}{\Gamma(k^2/2 + 1)} \le \left(\frac{\sqrt{2e\pi n}K\eps}{\sqrt k}\right)^{k^2} \label{eq:m11-bound}
\end{align}
In the final two steps we have used the volume of a Frobenius norm ball in $\R^{k\times k}$, and Stirling's approximation. Plugging into \eqref{eq:nk-conditional} gives
$$
    \P\left[\sigma_k(\bN_k) \ge 1/\eps\right] \le \P\left[\|\Im \bN_k^{-1}\|_F \le \eps\sqrt k\right] \left(\frac{\sqrt{2e\pi n}K\eps}{\sqrt k}\right)^{k^2},
$$
and we now turn to the more serious task of the requisite small-ball probability estimate for $\Im \bN_k^{-1}$. This calculation is facilitated by a second key observation, which is an immediate consequence of the full expression $\eqref{eq:imagpart}$ for $\Im \bN_k^{-1}$.

\begin{observation}
    Conditional on $\bM_{22}$, the imaginary part $\Im \bN_k^{-1}$ is a quadratic function in $\bM_{12}$ and $\bM_{21}$, of the type studied in Section \ref{sec:anticoncentration}.
\end{observation}

In particular, for any deterministic $(n-k)\times(n-k)$ matrices $Y$ and $X$, and $j$ satisfying $n-k \ge j > k^2 + k + 1$, Theorem \ref{thm:anticoncentration} implies 
\begin{align}
    &\P\left[\|U_{12} - \bM_{12}Y\bM_{21} - \bM_{12} X U_{21} - U_{12} X \bM_{21} + U_{12}Y U_{21}\|_F \le \eps\sqrt k \right] \nonumber \\
    &\qquad \qquad \qquad \qquad \leq (1 + k^2)\left(\frac{\crv^2K^2 n \sqrt{2e\pi k}}{\sqrt{j-k+1}\sigma_j(\bY)}\right)^{k^2} \left(\frac{\sqrt{2e\pi}\eps}{\sqrt k}\right)^{k^2} \nonumber \\
    &\qquad \qquad \qquad \qquad = (1 + k^2)\left(\frac{\crv^2K^2n\cdot 2e\pi}{\sqrt{j-k+1}\sigma_j(\bY)}\right)^{k^2}, \label{eq:deterministicboundanitcon}
\end{align}
(again using the volume of a Frobenius norm ball). Since $\bY$ depends only on the randomness in $\bM_{22}$, and is thus independent of $\bM_{12}$ and $\bM_{21}$, conditioning and integrating over $\bM_{22}$ gives us
\begin{equation}
    \label{eq:imag-bound}
    \P\left[\|\Im \bN_k^{-1} \| \le \eps \right] \le (1 + k^2)\left(\frac{\crv^2K^2n \cdot 2e\pi}{\sqrt{j-k+1}}\right)^{k^2} \E\left[\sigma_j(\bY)^{-k^2}\right].
\end{equation}

To finish the proof, we now need to bound this remaining expectation for a suitable choice of $j$, satisfying $n-k \ge j > k^2 + k + 1$. In \eqref{eq:defofXY}, we defined $\bY = \Im (\bM_{22} + i U_{22})^{-1}$, and we now require a more explicit formula. Using the representation of $\C^{(n-1)\times(n-1)}$ as a set of block matrices in $\R^{2(n-1)\times 2(n-1)}$, and again applying the Schur complement formula,
$$
    \begin{pmatrix}
        \bX & -\bY  \\
        \bY & \bX
    \end{pmatrix} =
    \begin{pmatrix}
       \bM_{22}  & -U_{22}  \\
       U_{22}  & \bM_{22}
    \end{pmatrix}^{-1} = \begin{pmatrix}
      (\bM_{22} + U_{22}\bM_{22}^{-1}U_{22})^{-1} & (\bM_{22} + U_{2,2}\bM_{22}^{-1}U_{22})^{-1}U_{22} \bM_{22}^{-1}  \\
     -(\bM_{22} + U_{22}\bM_{22}^{-1}U_{22})^{-1}U_{22} \bM_{22}^{-1} & (\bM_{2,2} + U_{22}\bM_{22}^{-1}U_{22})^{-1}
    \end{pmatrix} 
$$
and hence 
\begin{equation}
    \label{eq:XYformulas}
     \bY =  -(\bM_{22} + U_{22}\bM_{22}^{-1} U_{22} )^{-1}U_{22} \bM_{22}^{-1}.
\end{equation}
If we could invert $U_{22}$, we could rewrite this as $-(\bM_{22}U_{22}^{-1}\bM_{22} + U_{22})^{-1}$ and set $j = n-k$, giving
$$
    \sigma_{n-k}\left(\bY\right)^{-k^2} = \|\bM_{22}U_{22}\bM_{22} + U_{22}\|^{k^2} \le \left(|\delta|^{-1}\|\bM_{22}\|^2 +|\delta|\right)^{k^2} \le \left(|\delta|^{-1}\|\bM_n\|^2 + |\delta|\right)^{k^2}.
$$
However, not every principal block of a permutation matrix is invertible, so we will need to work a bit harder.
 
 

Since $U$ is a permutation matrix, and $U_{22}$ is an $(n-k)\times (n-k)$ block of $\delta U$, by the usual interlacing of singular values for submatrices \cite[Corollary 7.3.6]{horn2012matrix}, we can be sure that $\sigma_1(U_{22})= \cdots = \sigma_{n-2k}(U_{22})=|\delta|.$ Hence, there exists a matrix $E$ of rank at most $2k$ such that $\hat{U}_{22} \defeq U_{22}+E$ is invertible, with all singular values equal to $|\delta|$. We can therefore write
$$
    \bY =-(\bM_{22} + U_{22}\bM_{22}^{-1} U_{22} )^{-1}U_{22} \bM_{22}^{-1} = -(\bM_{22} + \hat{U}_{22}\bM_{22}^{-1} U_{22} + \bE_1)^{-1}\hat{U}_{22} \bM_{22}^{-1}+\bE_2  
$$
where $\bE_1 = -E\bM_{22}^{-1}U_{22}$ and $\bE_2=-(\bM_{22}- U_{22}\bM_{22}^{-1} U_{22} )^{-1}E\bM_{22}^{-1}$. Since $\rank(\bE_2)\leq \rank(E) \leq 2k$, interlacing of singular values upon low-rank updates \cite[Theorem 1]{thompson1976behavior} ensures
    \begin{equation}
    \label{eq:eq21}
    \sigma_{j}(\bY)\geq \sigma_{j+2k}\left((\bM_{22}+ \hat{U}_{22}\bM_{22}^{-1} U_{22} +\bE_1)^{-1}\hat{U}_{22} \bM_{22}^{-1}\right).
\end{equation}
On the other hand 
\begin{equation}
    \label{eq:eq22}
    (\bM_{22}+ \hat{U}_{22}\bM_{22}^{-1} U_{22} +\bE_1)^{-1}\hat{U}_{22} \bM_{22}^{-1} = (\bM_{22}\hat{U}_{22}^{-1}\bM_{22}+ U_{22}+M_{22}\hat{U}_{22}^{-1}\bE_1)^{-1}, 
\end{equation}
and since $\rank(\bM_{22} \hat{U}^{-1}_{22}\bE_1)\leq \rank(\bE_1)\leq \rank(E)\leq 2k$, a further application of the low-rank update bound tells us
\begin{equation}
    \label{eq:eq23}    
    \sigma_{j+2k}\left((\bM_{22}\hat{U}_{22}^{-1}\bM_{22}+ U_{22}+\bM_{22}\hat{U}_{22}^{-1}\bE_1)^{-1} \right)\geq \sigma_{j+4k} \left((\bM_{22}\hat{U}_{22}^{-1}\bM_{22}+ U_{22})^{-1}\right). 
\end{equation}

Putting together \eqref{eq:eq21}, \eqref{eq:eq22}, and \eqref{eq:eq23}, we get
$$
    \sigma_j(\bY)\geq \sigma_{j+4k} \left((\bM_{22}\hat{U}_{22}^{-1}\bM_{22}+ U_{22})^{-1}\right),
$$
and finally, setting $j = n-5k$, and recalling $\|U_{2,2}\| = |\delta|$, $\|\hat U_{2,2}^{-1}\| = |\delta|^{-1}$, and $\|\bM_{22}\| \le \|\bM_n\|$, we have
\begin{align}
    \sigma_{n-5k}(\bY)^{-k^2} 
    \leq \left\|\bM_{22}\hat{U}_{22}^{-1}\bM_{22}+ U_{22}\right\|^{k^2}
    \leq \left(|\delta|^{-1}\|\bM_n\|^2 + |\delta|\right)^{k^2} \label{eq:sigmajY-ub}
\end{align}
We now assemble our work so far:
\begin{align*}
    \P\left[\sigma_k(\bN_k) \ge 1/\eps\right] 
    &\le \P\left[\|\Im \bN_k^{-1}\| \le \eps\right] \left(\frac{\sqrt{2e\pi n} K \eps}{\sqrt k}\right)^{k^2} \\
    &\le (1 + k^2)\left(\frac{\crv^2 K^2n \cdot 2e\pi \cdot \eps }{\sqrt{j-k+1}}\right)^{k^2}\left(\frac{\sqrt{2e\pi n} K \eps}{\sqrt k}\right)^{k^2}\E\left[\sigma_j(\bY)^{-k^2}\right] & & \forall n-k \ge j \ge k^2 + k + 1 \\
    &\le (1 + k^2)\left(\frac{\crv^2 K^3 (2e\pi n)^{3/2} }{\sqrt{k(n-6k+1)}}\right)^{k^2}\left(\frac{\eps^2}{|\delta|}\right)^{k^2}\E\left(\|\bM_n\| + \delta^2\right)^{k^2} & & \text{setting $j=n-5k$}.
\end{align*}
For this to go through, we need $n \ge \max\{6k,(k+2)^2\} = (k+2)^2$. Finally, we can use $1/(n-6k+1) \le 6k/n$ to obtain the final result.

\subsection{Proof of Theorem \ref{prop:gaussiansingvalscomplex}}
\label{subsec:gaussproof}

We will first modify the proof of Lemma \ref{lem:taildiagonalentries}, referring back to the argument in the prior section. In order to perform these modifications, set $K = 1/\gamma$, and think of $\bM_n = K^{-1}\bG_n$. As above, $\delta \in \R$ is a real number, $U$ is a permutation, and we write $\bN_k$ for the upper left $k\times k$ block of $(\delta i U - \bM_n)^{-1}$. In \eqref{eq:m11-bound}, using that the density of each entry of $\bM_n$ is bounded by $(2\pi)^{-1/2}K\sqrt n$, we find
$$
    \E\left[\indicator{\|\Re N_k^{-1}\| \le \eps} \md M_{12},M_{21}M_{22}\right] \le \left(\frac{\sqrt{en} K \eps}{\sqrt k}\right)^{k^2}.
$$
In \eqref{eq:deterministicboundanitcon} and \eqref{eq:imag-bound}, swapping Theorem \ref{thm:anticoncentration-gaussian} for Theorem \ref{thm:anticoncentration}, we have that for any $n-k \ge j > 2k$,
$$
    \P\left[\|\Im \bN_k^{-1}\| \le \eps\right] \le \left(\frac{K^2 n}{2\sqrt{j-2k+1}}\right)^{k^2}\E \left[\sigma_j(\bY)^{-k^2}\right];
$$
finally, in $\eqref{eq:sigmajY-ub}$ if we now set $j = n-5k$, we have
$$
    \E\left[\sigma_{n-5k}(\bY)^{-k^2}\right] \le \E\left[\left(|\delta|^{-1}\|M\|^2 + |\delta|\right)^{k^2}\right].
$$
Putting all this together, for any $k$ satisfying $n\ge 7k$,
\begin{align}
    \label{eq:diagonal-block-gaussian}
    \P\left[\sigma_k(\bN_k)\ge 1/\eps\right] &\le \left(\frac{\sqrt{7e} K^3n}{2}\frac{\eps^2}{|\delta|}\right)^{k^2}\E\left[\left(\|\bM_n\|^2 + \delta^2\right)^{k^2}\right].
\end{align}
Now, let $z\in \C$, and continue as in the proof of Theorem \ref{thm:singvalscomplexshifts} from Lemma \ref{lem:taildiagonalentries}. Recalling $K = 1/\gamma$, and substituting \eqref{eq:diagonal-block-gaussian} in place of \eqref{eq:tailbound11entry}, we obtain
\begin{align*}
    \P\left[\sigma_k(z - A - \gamma \bG_n) \le \eps \right] 
    &\le {\binom{n}{k}}^2 \left(\frac{\sqrt{7e} k^2 n^3}{2\gamma^3} \left(\left(\gamma B_{\bG_n,2k^2} + \|A\| + |\Re z|\right)^2 + |\Im z|^2\right)\frac{\eps^2}{|\Im z|}\right)^{k^2}.
\end{align*}

\section{Lower Bounds on the Minimum Eigenvalue Gap}
\label{sec:gaps}

This section is devoted to several results regarding eigenvalue gaps of real random matrices with independent entries. Below we state the main result of this section. 

\begin{rtheorem}{\ref{thm:gaps}}
\label{thm:mainofgaps}
    Let $n\geq 16$, $A \in \R^{n\times n}$ be deterministic, and $\bM_n$ be a random matrix satisfying Assumption \ref{assumption} with parameter $K>0$. For any $0 < \gamma < K$ and $R>1$,
    $$
        \P\left[\gap(A + \gamma \bM_n) \le s \right] \le C_{\ref{thm:gaps}} R^2\left(\gamma B_{\bM_n,8} + \|A\| + R\right) (K/\gamma)^{5/2} n^4 s^{1/3} + \P\left[\|A + \bM_n\| \ge R\right],
    $$
    where $C_{\ref{thm:gaps}}$ is a universal constant defined in equation \eqref{eq:gaps-const-def}. Moreover, if $\bG_n$ is an $n\times n$ real Ginibre and $0 < \gamma < 1$, then
    $$
        \P\left[\gap(A + \gamma \bG_n) \le s \right] \le  15 \left(\|A\| + 7\right)^3 n^3\gamma^{-5/2}s^{1/3} + e^{-2n}.
    $$
\end{rtheorem}

As discussed in the introductory material, our proof of this theorem hinges on a deterministic fact: one can detect eigenvalues of a matrix $M$ close to a point $z \in \C$ simply by studying the smallest two singular values of $z - M$. This fact is a direct consequence of the log-majorization theorem \cite[Theorem 3.3.2]{horn1994topics}, which implies $\sigma_n \sigma_{n-1} \le |\lambda_n \lambda_{n-1}|$ for any matrix. 
We now state this carefully.

\begin{lemma} \label{lem:gaps}
    Let $M \in \C^{n\times n}$ be any complex matrix and $z\in \C$. If $M$ has two eigenvalues in $D(z, r)$, then $$\sigma_n(z - M)\sigma_{n-1}(z - M) \le r^2.$$
\end{lemma}


 Hence, in order to obtain a tail bound for eigenvalue gaps of a random matrix, it is enough to obtain appropriate tail bounds for the two smallest singular values of its shifts. We use $D(z_0,r)=\{z\in \C:|z-z_0|\le r\}$ to denote a closed disk in the complex plane. 

\begin{proof}[Proof of Theorem \ref{thm:gaps} ]
    For most of the proof, let us absorb $\gamma$ into the constant $K$---the condition $\gamma < 1/K$ will not be relevant until the end. Lemma \ref{lem:gaps} in hand, we will use a simple net argument: choose a covering of the region $D(0,R) \subset \C$ with disks, with the property that any pair of eigenvalues at distance less than $s$ must both lie in at least one of them. Our only complication is that our tail bounds on the singular values of $z - (A + \bM_n)$ depend on the shift $z$: on the real line they are governed by Theorem \ref{thm:skboundgeneral} , and away from it by Theorem \ref{thm:singvalscomplexshifts}. 
    
    To handle this, we will use a somewhat elaborate combination of nets, exploiting the fact that real matrices have conjugate-symmetric spectra. Specifically, this symmetry means that we can think of small gaps as arising in one of three different ways: gaps in which at least one eigenvalue is real, gaps between a conjugate pair of eigenvalues with small imaginary part, and gaps between complex eigenvalues away from the real line. Thus motivated, let us define, for any matrix $M \in \R^{n\times n}$ and $\delta > 0$,
    \begin{align*}
        \gap_{\R}(M) &\defeq \min \left\{\left|\lambda_i(M) - \lambda_j(M)\right| : i\neq j \text{ and } \lambda_i(M) \in \R \right\} \\
        \Im_{\min}(M) &\defeq \min\left\{|\Im \lambda_i(M)| : \lambda_i(M) \notin \R \right\} \\
        \gap_{\Im \ge \delta}(M) &\defeq \min \left\{ \left|\lambda_i(M) - \lambda_j(M) \right| : i\neq j \text{ and } |\Im \lambda_i(M)|,|\Im \lambda_j(M)| \ge \delta \right\},
    \end{align*}
    and observe that if $\delta > 0$, 
    $$
        \{\gap(A + \bM_n) \le s \} = \{\gap_{\R}(A + \bM_n) \le s\} \cup \{\Im_{\min}(A + \bM_n) \le \delta\} \cup \{\gap_{\Im \ge \delta}(A + \bM_n) \le s\}.$$
    We will set up a separate net to union bound each of these events: let
    \begin{align*}
        \calN_\eta^{\R} &\defeq \{j\eta : j \in \mathbb{Z}\} \cap [-R,R] \\
        \calN_{\delta,\eta}^{\C} &\defeq \{\eta j + i(\delta + \eta k) : j, k \in \mathbb{Z} \} \cap B(0,R).
    \end{align*}    
    Then, judiciously choosing the spacing and radii of disks, for any $\delta>0$ we have:
    \begin{equation}
    \label{eq:union-bound-first}
    \begin{aligned}
        \P\left[\gap(A + \bM_n) \le s \right]
        &\le \sum_{z \in \calN_{2s}^{\R}}\P\left[|\Lambda(A + \bM_n) \cap D(z,3s/2)| \ge 2 \right] \\
        &\qquad + \sum_{z \in \calN_{\delta}^{\R}}P\left[|\Lambda(A + \bM_n) \cap D(z,\sqrt{2}\delta)| \ge 2 \right] \\
        &\qquad + \sum_{z \in \calN_{\delta,s}^{\C}}\P\left[\Lambda(A + \bM_n) \cap D(z,\sqrt{5/4} s)| \ge 2 \right] \\
        &\qquad + \P\left[\|A + \bM_n\| \ge R \right].
    \end{aligned}
    \end{equation}
    The first line controls $\gap_{\R}$, the second one $\Im_{\min}$, the third one $\gap_{\Im \ge \delta}$, and the final one the event that some eigenvalue lies outside the region covered by our net. One could further optimize the above in the pursuit of tighter constants, but we optimize for simplicity. The remainder of the proof consists of bounding these events with Theorems \ref{thm:skboundgeneral} and \ref{thm:singvalscomplexshifts}---the constants and exponents become somewhat unwieldy, and on a first reading we recommend following the argument at a high level to avoid being bogged down in technicalities. The Gaussian case is quite similar, and we defer it to Appendix \ref{sec:gaussian}.
\bigskip

\noindent \emph{Step 1: Gaps on the Real Line.} We first must bound the probability
$$
    \P\left[|\Lambda(A + \bM_n) \cap D(z,3s/2)| \ge 2\right]
$$
for $z \in \R$. To use Lemma \ref{lem:gaps}, we need tail bounds for the \textit{product} of the two smallest singular values of $z - (A + \bM_n)$, whereas Theorem \ref{thm:skboundgeneral} concerns individual singular values. To get around this, note that for every $z \in \R$ and $x > 0$,
\begin{align*}
    \P\left[\sigma_n(z - (A + \bM_n))\sigma_{n-1}(z - (A + \bM_n)) \le r^2 \right]  
    &\le \P\left[\sigma_n(A + \bM_n) \le rx\right] + \P\left[\sigma_{n-1}(A + \bM_n) \le r/x\right] \\
    &\le \sqrt{2} \crv K n^2 rx + 4\crv^4 K^4 n^6 r^4/x^4.
\end{align*}
Optimizing in $x$, we have
\begin{equation}
    \label{eq:szarekbound}
    \P\left[|\Lambda(A + \bM_n) \cap D(z,r)|\ge 2 \right] \le \left(4^{1/5} + 4^{-4/5}\right)\left(\sqrt 2 \crv K r\right)^{8/5}n^{14/5} \le 3n^{14/5}(\crv K r)^{8/5}.
\end{equation}
The rough bound $\left|\calN_{2s}^{\R}\right| \le (R/s + 1) \le 3R/2s$ now gives
\begin{align}
    \sum_{z \in \mathcal{N}_s^{\R}}\P\left[|\Lambda(A + \bM_n) \cap D(z, 3s/2)| \ge 2 \right] \nonumber 
    &\le \left|\mathcal{N}_s^{\R}\right| \cdot 3n^{14/5}(3\crv K s/2)^{8/5} \nonumber \\
    &\le 9R(\crv K)^{8/5}n^{14/5}s^{3/5} \label{eq:realgaps}
\end{align}

\bigskip

\noindent \emph{Step 2: Eigenvalues Near the Real Line.} 
Using \eqref{eq:szarekbound} and imitating the remainder of Step 1,
\begin{equation}
    \label{eq:eigsnearrealline}
    \sum_{z \in \calN_{\delta}^{\R}}P\left[|\Lambda(A + \bM_n) \cap D(z,\sqrt{2}\delta)| \ge 2 \right] \le 8 R (\crv K )^{8/5}n^{14/5} \delta^{3/5}
 \end{equation}
This directly implies a stand-alone tail bound on $\Im_{\min}$, which we record for use in Section \ref{sec:regularization},:
\begin{equation}
    \label{eq:im-min-standalone}
    \P\left[\Im_{\min}(A + \bM_n) \le \delta\right] \le 8 R (\crv K )^{8/5}n^{14/5} \delta^{3/5} + \P[\|\bM_n\|\geq R].
\end{equation}

\bigskip

\noindent \emph{Step 3: Eigenvalues Away from the Real Line.} We finally turn to non-real $z$. As in Step 1, observe that for any $z\in \C\setminus \R$, $r>0$, and $n \ge 16$, Theorem \ref{thm:singvalscomplexshifts} implies
\begin{align}
    \P\left[|\|\Lambda(A + \bM_n) \cap D(z,r)| \ge 2 \right] 
    &\le \min_{x >0}\left\{ \P\left[\sigma_n(A + \bM_n) \le rx\right] + \P\left[\sigma_{n-1}(A + \bM_n) \le r/x\right]\right\} \nonumber \\
    &\le \min_{x > 0} \left\{2C_{\ref{thm:singvalscomplexshifts}}K^3n^5\left(\left(B_{\bM_n,2} + \|A\| + |\Re z|\right)^2 + |\Im z|^2\right) \frac{(rx)^2}{|\Im z|} \right. \nonumber \\
    &\qquad\qquad \left. + 640 C_{\ref{thm:singvalscomplexshifts}}^4 K^{12}n^{14}\left(\left(B_{\bM_n,8} + \|A\| + |\Re z|\right)^2 + |\Im z|^2\right)^4 \frac{r^8}{x^8|\Im z|^4}\right\} \nonumber \\
    &\le C_{(\ref{eq:disksaway})} \left(\frac{\left(B_{\bM_n,8} + \|A\| + |\Re z|\right)^2 + |\Im z|^2}{|\Im z|}\right)^{8/5} K^{24/5}r^{16/5} n^{34/5} \label{eq:disksaway}
\end{align}
where we have used $B_{\bM_n,1} \le B_{\bM_n,8}$ and defined $C_{\eqref{eq:disksaway}} = 11 C_{\ref{thm:singvalscomplexshifts}}$.

Finally, observing that every $z \in \mathcal{N}_{\delta,s}^{\C}$ has $|\Im z| > \delta$ and $|z| \le R$, we have
\begin{align}
    \sum_{z\in \mathcal{N}_{\delta, s}^\C} & \P\left[|\Lambda(A + \bM_n) \cap D(z, \sqrt{5/4}s)|\geq 2\right] \nonumber \\
    &\qquad \leq 6(R/s)^2 C_{\eqref{eq:disksaway}}\left(\frac{(B_{\bM_n,8} + \|A\| + R)^2}{\delta}\right)^{8/5}K^{24/5}(\sqrt{5}s/2)^{16/5}n^{34/5} \nonumber \\
    &\qquad \leq C_{\eqref{eq:eigsaway}}R^2(B_{\bM_n,8} + \|A\| + R)^{16/5} \frac{K^{24/5}s^{6/5}n^{34/5}}{\delta^{8/5}} \label{eq:eigsaway}
\end{align}
where $C_{\eqref{eq:eigsaway}} \defeq 6(5/4)^{8/5} C_{\eqref{eq:disksaway}}$.
  
\bigskip

\noindent \emph{Step 4: Conclusion.} 
We now put together the three steps above, substituting \eqref{eq:realgaps}, \eqref{eq:eigsnearrealline}, and \eqref{eq:eigsaway} into \eqref{eq:union-bound-first}, and adding back in the $\gamma$ scaling. Using the fact that $\psi \delta^{3/5} + \phi s^{6/5}\delta^{-8/5} \le 2\psi^{8/11}\phi^{3/11}s^{18/55}$, we obtain
\begin{align}
    \P\left[\gap(A + \bM_n) \le s\right]
    &\le 9R(\crv K/\gamma)^{8/5} n^{14/5} s^{3/5} \nonumber \\ 
        & + 2\left(C_{\eqref{eq:eigsaway}}R^2(\gamma B_{\bM_n,8} + \|A\| + R)^2 (K/\gamma)^{24/5}n^{34/5}\right)^{3/11}\left(8(\crv K/\gamma)^{8/5}n^{14/5}\right)^{8/11} s^{18/55} \nonumber \\
        & + \P\left[\|A + \gamma \bM_n\| \ge R\right] \nonumber \nonumber \\
    &\le C_{\ref{thm:gaps}} R^{14/11} \left(\gamma B_{\bM_n,8} + \|A\| + R\right)^{6/11} (K/\gamma)^{136/55} n^{214/44} s^{18/55} + \P\left[\|A + \bM_n\| \ge R\right] \nonumber \\
    &\le C_{\ref{thm:gaps}} R^2\left(\gamma B_{\bM_n,8} + \|A\| + R\right) (K/\gamma)^{5/2} n^4 s^{1/3} + \P\left[\|A + \bM_n\| \ge R\right] , \label{eq:gapfinal}
\end{align}
where 
\begin{equation}
    \label{eq:gaps-const-def}
    C_{\ref{thm:gaps}} = 2C_{(\ref{eq:eigsaway})}^{3/11}\cdot 8^{8/11} \crv^{64/55} + 9\crv^{8/5}.
\end{equation}

\end{proof}

\section{Upper Bounds on the Eigenvalue Condition Numbers} \label{sec:regularization} In this section, we convert our probabilistic lower bounds on the least singular value into upper bounds on the mean eigenvalue condition numbers, following \cite{banks2019gaussian}.
The following fact is elementary; a proof appears in \cite{banks2019gaussian}.
\begin{lemma}[Limiting Area of Pseudospectrum]
\label{prop:limiting-area-complex}
    Let $M$ be an $n \times n$ matrix with $n$ distinct eigenvalues $\lambda_1, \dots, \lambda_n$.   Let $\vol_\C$ denote the Lebesgue measure on $\mathbb{C}$, and let $\Omega\subset \C$ be a measurable open set.  Then
    \[ \liminf_{\eps \to 0} \frac{\vol_\C(\Lambda_\eps(M)\cap \Omega)}{\eps^2} \ge \pi \sum_{\lambda_i\in \Omega}^n \kappa(\lambda_i)^2.\] 
\end{lemma}
 
In addition to Lemma \ref{prop:limiting-area-complex}, we will need an easy variant relating pseudospectrum on the real line to the conditon numbers of \textit{real} eigenvalues.

\begin{lemma}[Limiting Length of Pseudospectrum on Real Line] \label{lem:limiting-area-real}
    Let $M \in \R^{n\times n}$ have $n$ distinct eigenvalues $\lambda_1,...,\lambda_n$.  Let $\vol_\R$ denote the Lebesgue measure on $\mathbb{R}$, and let $\Omega \subset \R$ be a measurable open set. Then
    $$
        2\sum_{\lambda_i \in \Omega} \kappa(\lambda_i) \le \liminf_{\eps \to 0} \frac{\vol_\R\left(\Lambda_\eps(M) \cap \Omega \right)}{\eps}
    $$
\end{lemma}

\begin{proof}
    For each $z \in \mathbb{C}$ and $r \ge 0$, let $D(z, r)$ denote the closed disk centered at $z$ of radius $r$. In the proof of \cite[Lemma 3.2]{banks2019gaussian} it is shown that if $M$ has $n$ distinct eigenvalues,
    \[ \bigcup_{i=1}^n D(\lambda_i,   \kappa(\lambda_i) \eps - O(\eps^2)) \subseteq \Lambda_{\eps}(M) \subseteq \bigcup_{i=1}^n D(\lambda_i,  \kappa(\lambda_i) \eps + O(\eps^2)).\]
    In particular, each $\lambda_i \in \Omega$ contributes at least $2 \kappa(\lambda) \eps - O(\eps^2)$ to the measure of $\Lambda_\eps \cap \Omega$.  Taking $\eps \to 0$ yields the conclusion.
\end{proof}
In both Lemma \ref{prop:limiting-area-complex} and Lemma \ref{lem:limiting-area-real}, if the boundary of $\Omega$ contains none of the eigenvalues, one actually has equality, the limit inferior can be replaced by the limit, and $\Omega$ need not be open, but we will not need this fact. 

\subsection{Bounds in Expectation}
We now come to the first main proposition of this section.
\begin{proposition}[$\kappa(\lambda_i)$ on the real line] \label{thm:realkappaibound} 
    Let $A \in \R^{n\times n}$ be deterministic, and let $\bM_n$ satisfy Assumption \ref{assumption} with parameter $K>0$. Write $\blambda_1,...,\blambda_n$ for the eigenvalues of $A + \gamma \bM_n$. Then for every measurable open set $\Omega \subset \R$,
    $$
        \E\sum_{\blambda_i \in \Omega}\kappa(\blambda_i) \le \frac{\crv K n^2}{2\gamma} \cdot \vol_\R(\Omega).
    $$
    In the case where $\bM_n$ is real Ginibre, one has the improvement
    $$
        \E\sum_{\blambda_i \in \Omega}\kappa(\blambda_i) \le \frac{n}{2\gamma} \cdot \vol_\R (\Omega).
    $$
    
\end{proposition}
\begin{proof}
    When $z$ is real, $z - A$ is also real, so we may apply the tail bound in Corollary \ref{cor:realrealsingulargamma}.  In particular, setting $k = 1$, we obtain the following tail bound for real $z$:
     \[\P[\sigma_{n}((z - A) + \gamma(-\bM_n) ) \le \eps] < \frac{\crv K n^2 \eps}{\gamma} .\]
     Since the eigenvalues of $z - (A + \gamma \bM_n)$ are distinct with probability 1, we have
     \begin{align*}
        2 \E \sum_{\lambda_i \in \Omega}\kappa(\lambda_i) 
        &\le \E\liminf_{\eps \to 0} \eps^{-1}\vol_\R\left(\Lambda_\eps(A + \gamma \bM_n) \cap \Omega \right) & & \text{Lemma \ref{lem:limiting-area-real}} \\
        &\le \liminf_{\eps \to 0} \eps^{-1} \E\int_{\Omega} \mathbf{1}_{\{z\in \Lambda_\eps(A + \gamma \bM_n)\}} \,dz & & \text{Fatou's lemma}\\
        &= \liminf_{\eps \to 0} \eps^{-1} \int_\Omega \P[z \in \Lambda_\eps(A + \gamma \bM_n)] \,dz & & \text{Fubini's theorem}\\
        &= \liminf_{\eps \to 0}\eps^{-1} \int_{\Omega} \P[\sigma_n(z - (A + \gamma \bM_n)) < \eps]\,dz & & \\
        &\le \frac{\crv K n^2 }{\gamma}  \vol_\R(\Omega). & & \text{Corollary \ref{cor:realrealsingulargamma}}
    \end{align*} 
    To obtain the improvement in the Ginibre case, in the final inequality we use the bound \[\P[\sigma_{n}(z - (A + \gamma \bG_n)) \le \eps] \le \frac{n \eps}{\gamma}\] instead, from Theorem \ref{thm:gaussianrealrealsingular}.  
\end{proof}




We now give the analogous proposition for the nonreal eigenvalues.

\begin{proposition}[$\kappa(\lambda_i)$ away from real line] \label{prop:nonrealkappaibound}
   Let $n \ge 9$.  Let $A \in \R^{n\times n}$ be deterministic.  Let $\bM_n$ satisfy Assumption \ref{assumption} with parameter $K>0$.  Let $\gamma > 0$, and write $\blambda_1,...,\blambda_n$ for the eigenvalues of $A + \gamma \bM_n$. Then for every measurable open set $\Omega \subseteq \C \setminus \R$,
    $$
        \E\sum_{\blambda_i \in \Omega}\kappa(\blambda_i)^2 \le \frac{C_{\ref{thm:singvalscomplexshifts}} K^3 n^{5}}{\gamma^3} \int_\Omega \frac{ (\gamma \E \Vert \bM_n \Vert + \Vert A \Vert + | \Re z | )^2 + |\Im z |^2 }{ |\Im z |}  \,dz.
    $$
    In the special case where $\bM_n$ is real Ginibre, one may take $n \ge 7$ and replace the term $C_{\ref{thm:singvalscomplexshifts}}K^3$ with $\frac{\sqrt{7e}}{4\pi}$. 
\end{proposition}
\begin{proof}
    In the proof of Theorem \ref{thm:realkappaibound}, since $\Omega \subseteq \C \setminus \R$ we replace Lemma \ref{lem:limiting-area-real} with Lemma \ref{prop:limiting-area-complex}. Since $z$ is no longer real we must also replace the singular value tail bound in Corollary \ref{cor:realrealsingulargamma} with the one in Theorem \ref{thm:singvalscomplexshifts} (or the one in Theorem \ref{prop:gaussiansingvalscomplex}, for the Ginibre case).
\end{proof}

\subsection{Bounds with high probability: Proofs of Theorems \ref{thm:kappai-probabilistic-intro} and \ref{thm:gaussiankappai-probabilistic-intro}}
We now prove the main theorem of this section, which implies that all eigenvalue condition numbers are bounded by $\poly(n/\gamma)$ with probability $1 - 1/\poly(n)$.  In the notation of the theorem below, $R, \Vert A \Vert, K$, and $\gamma$ will be $\Theta(1)$ in most applications, so $\eps_1$ and $\eps_2$ may be set to $1/n^D$ for sufficiently high $D$.

\begin{rtheorem}{\ref{thm:kappai-probabilistic-intro}}
    Let $n \ge 9$.  Let $A \in \R^{n\times n}$ be deterministic, and let $\bM_n$ satisfy Assumption \ref{assumption} with parameter $K>0$. Let $0 < \gamma < K \min\{1, \Vert A \Vert + R\}$, and write $\blambda_1,...,\blambda_n$ for the eigenvalues of $A + \gamma \bM_n$.  Let $R > \E \Vert \gamma \bM_n \Vert$.  Then for any $\eps_1, \eps_2 > 0$, with probability at least $1 - 2\eps_1 -  O\left(\frac{R(R + \Vert A \Vert)^{3/5}K^{8/5}n^{14/5} \eps_2^{3/5}}{\gamma^{8/5}}\right) - 2 \P[\gamma \Vert \bM_n \Vert > R]$ we have
    \[ 
        \sum_{\blambda_i \in \R} \kappa(\blambda_i) \le \eps_1^{-1} C_{\ref{thm:kappai-probabilistic-intro}}  K n^2 \frac{\Vert A \Vert + R}{\gamma}, 
    \]
    \[ 
        \sum_{\blambda_i \in \C \setminus \R} \kappa(\blambda_i)^2 \le \eps_1^{-1} \log(1/\eps_2) C_{\ref{thm:kappai-probabilistic-intro}}  K^3 n^{5} \cdot \frac{(\Vert A \Vert + R)^3}{\gamma^3},\qquad\text{and}
    \]
    \[
        \kappa_V(A + \gamma \bM_n) \le \eps_1^{-1} \sqrt{\log(1/\eps_2)} C_{\ref{thm:kappai-probabilistic-intro}}  K^{3/2} n^3 \cdot \frac{(\Vert A \Vert + R)^{3/2}}{\gamma^{3/2}},
    \]
for some universal constant $C_{\ref{thm:kappai-probabilistic-intro}}  > 0$.
\end{rtheorem}

\begin{proof}
From here on out, assume that each of $\sum_{\blambda_i \in \R} \kappa(\blambda_i)$ and $\sum_{\blambda_i \in \C \setminus \R} \kappa(\blambda_i)^2$ is at most $\eps_1^{-1}$ times its expectation; by Markov's inequality and a union bound this happens with probability at least $1 - 2\eps_1$.

Let $\delta \in (0, R)$ be a small parameter to be optimized later.  Let $L := \Vert A \Vert + R$, and define the regions $\Omega_\mathbb{R}$ and $\Omega_\C$ as follows:
\[ 
    \Omega_\mathbb{R} := \{x \in \mathbb{R} : |x| < L\} 
\]
\[ 
    \Omega_\C := \{x + yi : x \in \mathbb{R} \text{ and } \delta < |y| < L. \} 
\]

Write $E_{\text{bound}}$ for the event that $\gamma \Vert \bM_n \Vert < R$ and let $E_{\text{strip}}$ denote the event that $\Im_{\min}(A+ \gamma \bM_n)> \delta$. Then with probability at least $1 - 2\eps_1 -  \P[E_\text{bound}] - \P[E_\text{strip}]$, all eigenvalues of $A + \gamma M_n$ are contained in $\Omega_\R \cup \Omega_C$, so 
\begin{align*}
    \sum_{\blambda_i \in \R} \kappa(\blambda_i) = \sum_{\blambda_i \in \Omega_\R} \kappa(\blambda_i) 
    \le \frac{\crv K n^2 }{ 2 \gamma}  \vol_\R(\Omega_\R)
    \le \frac{\crv K n^2 L}{\gamma}
\end{align*}
and 
\begin{align*}
    \sum_{\blambda_i \in \C \setminus \R} \kappa(\blambda_i)^2 &= \sum_{\blambda_i \in \Omega_\C} \kappa(\blambda_i)^2 \\
    &\le  \frac{C K^3 n^{5}}{\gamma^3} \int_{\Omega_\C} \frac{ (\gamma \E \Vert M_n \Vert + \Vert A \Vert + | \Re z | )^2 + |\Im z |^2 }{ |\Im z |}  \,dz \\
    &\le   2  \frac{C K^3 n^{5}}{\gamma^3}  \int_\delta^{L} \int_{-L}^{L}  \frac{ ( \gamma \E \Vert M_n \Vert + \Vert A \Vert + | x| )^2 +  y^2 }{ y}  \,dx\,dy \\
    &\le 2 \frac{C K^3 n^{5}}{\gamma^3} \int_\delta^{L} 2 L \frac{ (2L)^2 +  L^2 }{ y } \,dy \\
    &= 20 \frac{C K^3 n^{5}}{\gamma^3} L^3 (\ln L  + \ln (1/\delta)). \\
\end{align*}

Recall from (\ref{eq:im-min-standalone}) that \[\P[E_{\text{strip}}] = O( R  K^{8/5}n^{14/5} \delta^{3/5}/\gamma^{8/5}) + \P[\gamma \|\bM_n\|\geq R],\]
so setting $\delta = L \eps_2$ yields the result.

To obtain the bound on $\kappa_V$, first note that by the definition of $\kappa_V$ and the fact that the Frobenius norm upper bounds the operator norm, we have
\[
    \kappa_V(A + \gamma \bM_n) \le \sqrt{n \sum_{i=1}^n \kappa(\blambda_i)^2} \le \sqrt{n} \sqrt{\left( \sum_{\blambda_i \in \R} \kappa_i(\blambda_i) \right)^2 + \sum_{\blambda_i \in \C \setminus \R} \kappa(\blambda_i)^2}.
\]
For a more detailed argument see the proof of Lemma 3.1 in \cite{banks2019gaussian}. Substituting this inequality in the bounds above yields the advertised result.
\end{proof}

In the special case of Ginibre matrices, we will endeavor to give an explicit bound on the constant factors appearing in the proof of Theorem \ref{thm:kappai-probabilistic-intro} without being too wasteful.  We also save one factor of $n$ in the bound for real eigenvalues in comparison to Theorem \ref{thm:kappai-probabilistic-intro}.

\begin{rtheorem}{\ref{thm:gaussiankappai-probabilistic-intro}}
    Let $n \ge 7$.  Let $A \in \R^{n\times n}$ be deterministic, and let $\bG_n$ be a real Ginibre matrix. Let $0 < \gamma < \min\{1, \Vert A \Vert\}$, and write $\blambda_1,...,\blambda_n$ for the eigenvalues of $A + \gamma \bG_n$.  Then for any $\eps_1, \eps_2 > 0$, with probability at least $1 - 2 \eps_1 -  \frac{30 \|A\|^{8/5} n^{8/5}}{\gamma^{8/5}}\eps_2^{3/5} - 2 e^{-2n}$ we have
    \[ 
        \sum_{\blambda_i \in \R}  \kappa(\blambda_i) \le 5 \eps_1^{-1} n \frac{\Vert A \Vert}{\gamma},
    \]
    \[
        \sum_{\blambda_i \in \C \setminus \R} \kappa(\blambda_i)^2 \le
    1000  \eps_1^{-1} \log(1/\eps_2)\frac{n^5  \Vert A \Vert^3}{\gamma^3},\qquad\text{and}
    \]
    \[
        \kappa_V(A + \gamma \bM_n) \le  1000 \eps_1^{-1}\sqrt{\log(1/\eps_2)}\frac{n^3 \Vert A \Vert^{3/2}}{\gamma^{3/2}}. 
    \]
\end{rtheorem}

\begin{proof}
We identify the necessary modifications to the proof of Theorem \ref{thm:kappai-probabilistic-intro}.  First, set $R = 4 \gamma$, so that $\P[ \gamma \Vert \bG_n \Vert > R] < e^{-2n}$. The statement for real eigenvalues is then immediate, using the improvement for Ginibre matrices in Proposition \ref{thm:realkappaibound}.

Now we proceed to the bound for the nonreal eigenvalues. Take $\delta = \eps_2 \Vert A \Vert$, so that by (\ref{eq:im-min-standalone-gaussian})---forthcoming in the appendix---we have 
\[
    \P[E_{\text{strip}}]\le  6(\Vert A \Vert + 4 \gamma) \frac{ n^{8/5} \Vert A \Vert^{3/5} \eps_2^{3/5}}{ \gamma^{8/5}}  \le \frac{30 n^{8/5} \|A\|^{8/5} \eps_2^{3/5}}{\gamma^{8/5}},
\] where we use $\gamma < \Vert A \Vert$.  Recall $\E \Vert G_n \Vert \le 2$  (see \cite{aubrun2017alice}).  Replacing $C_{\ref{thm:singvalscomplexshifts-intro}}K^3$ with $\frac{\sqrt{7e}}{4\pi}$ as indicated in Proposition \ref{prop:nonrealkappaibound}, and computing the integral
\begin{align*}  \int_\delta^{L} \int_{-L}^{L}  \frac{ ( L + | x| )^2 +  y^2 }{ |y |}  \,dx\,dy  &= 
\frac{14}{3} L^3 (\log L + \log(1/\delta)) + L^3 -  L\delta^2 \\
&\le \frac{14}{3} L^3 (\log L + \log(1/\eps_2) - \log \Vert A \Vert ) + L^3,  \end{align*}
one obtains
\[\sum_{\blambda_i \in \C \setminus \R} \kappa(\blambda_i)^2 \le \frac{7\sqrt{7e}}{6\pi \gamma^3} n^5 (\Vert A \Vert + 4\gamma)^3 (\log (\Vert A \Vert + 4\gamma) + \log(1/\eps_2) - \log \Vert A \Vert + 3/14). \]
Using $\gamma < \Vert A \Vert$ and cleaning up the constants, we arrive at the form in the theorem statement.
\end{proof}





\section{Further Questions}\label{sec:conclusion}

There are a few natural directions to pursue.  One direction is to prove analogous results for more general perturbations $\bM_n$. It was speculated in \cite{banks2019gaussian} that low-rank matrices could regularize the eigenvalue condition numbers of any matrix, but this is false; see Appendix \ref{sec:boundedrank} for a discussion.  
Another question is: what can be said about the eigenvalue condition numbers for random matrices without continuous entries? Solving this question would require  essentially different ideas from those presented in this paper. More concretely, our proof technique requires  $$\lim_{\eps\to 0}\P[\sigma_n(z-(A+\bM_n))\leq \eps]=0,$$ and this may no longer hold  if the distributions of the entries of $\bM_n$ are allowed to be discrete.   A natural starting point is the case of i.i.d. $\pm 1$ entries:  
\begin{problem}
 Let $\bM_n$ be a matrix with independent Rademacher entries. For which deterministic matrices $A$ and which $\gamma >0$ does it hold, with high probability,  that $\kappa_V(A+\gamma \bM_n) = O(n^C)$ for some $C > 0$?
\end{problem}

With regards to the least singular value of complex shifts of real ensembles,  we posit the following possible improvement to Theorem \ref{prop:gaussiansingvalscomplex} in the dependence on $n$:
\begin{conjecture}
 Let $\bG_n$ be an $n\times n$  real Ginibre matrix. Then, for any constant $C>0$ there exists a constant $C'$ (depending on $C$ only) such that  for any $ \eps>0$ and  $z\in \mathbb{C}\setminus \mathbb{R}$ with $|z|\leq C$ it holds that
\begin{equation}
\label{eq:conj1}
\mathbb{P}\left[\sigma_n(z-\bG_n)\leq \eps \right]\leq \frac{C'n^2 \eps^2 }{|\mathrm{Im(z)}|}. 
\end{equation}
\end{conjecture}

Actually, we believe that a stronger conjecture is true. Namely, the bound in (\ref{eq:conj1}) should hold  even when $\bG_n$ is substituted by $A+\bG_n$, where $A\in \mathbb{R}^{n\times n}$ is deterministic. In this case $C'$ is also allowed to depend on $\|A\|$. 

Our next conjecture is that Szarek's bound for singular values of real Ginibre matrices in Theorem \ref{thm:szarek} holds, up to the value of the universal constant $C$, in the more general setting of matrices satisfying Assumption \ref{assumption}.   This would constitute an improvement of Theorem \ref{thm:skboundgeneral} in the dependence on $k$ and $n$. 

\begin{conjecture}
\label{conj:generalSzarek}
 Let $\bM_n$ be a real random matrix satisfying Assumption \ref{assumption} with parameter $K>0$ and perhaps with some moment assumptions on its entries. Then, there is a universal constant $C$ such that for any deterministic $A\in \mathbb{R}^{n\times n}$, it holds that
$$\mathbb{P}\left[ \sigma_{n-k+1}(A+\bM_n)\leq \frac{k\eps}{n} \right] \leq (CK\eps)^{k^2}.$$
\end{conjecture}

It is worth noting that Conjecture \ref{conj:generalSzarek} is known to be true when $k=1$. This was proven by Tikhomirov in \cite{tikhomirov2017invertibility} under weaker assumptions on the independence of the entries of $\bM_n$.  

 \subsection*{Acknowledgments} We thank Amol Aggarwal, Nick Cook, Hoi Nguyen, and Stanis{\l}aw Szarek  for helpful conversations. We thank Vishesh Jain for helpful conversations and for pointing out the important reference \cite{ge2017eigenvalue} to us at the IPAM reunion workshop of the program ``Quantitative Linear Algebra'', which was the starting point of this work.

\bibliographystyle{alpha}
\bibliography{regularization}

	\appendix

\section{Bounded-rank Perturbations} \label{sec:boundedrank}
In this section, we show that in contrast to Ginibre and general continuous perturbations, bounded-rank perturbations do not regularize the pseudospectrum of all matrices.  Precisely, we have the following result:

\begin{proposition}\label{prop:boundedrank}
Fix $r \in \mathbb{N}$.  Let $A \in \mathbb{C}^{n \times n}$ be any matrix, and let $B: = A \otimes I_r \in \mathbb{C}^{nr \times nr}$.  Then $\Lambda_\eps(B) \subseteq \Lambda_\eps(B + M)$ for all $\eps > 0$ and all matrices $M$ with rank at most $r-1$.
\end{proposition}
As a particular example, if we choose $A$ to be nondiagonalizable, then $B$ is nondiagonalizable, and we recover the known fact that $B+M$ is nondiagonalizable for any $M$ with rank at most $r-1$.  (A matrix $X$ is nondiagonalizable if and only if $\lim_{\eps \to 0} \vol_\C(\Lambda_\eps(X))/\eps^2 = \infty$.) Much more is known about the Jordan structure upon low-rank perturbations; see e.g. \cite{ran2012eigenvalues} and the references therein, including \cite{hormander1994remark, moro2003low, savchenko2004change}.

\begin{proof}[Proof of Proposition \ref{prop:boundedrank}]
    For each $z \in \mathbb{C}$, we have $z - B = (z - A) \otimes I_r$.  Thus, \[\sigma_n(z - B) = \sigma_{n-1}(z-B) = \dots = \sigma_{n - r + 1}(z - B) = \sigma_n(z-A).\]
    Then for any matrix $M \in \mathbb{C}^{n \times n}$ with rank at most $r - 1$, repeated application of interlacing for rank-one updates yields
    \[ \sigma_n(z - B - M) \le \sigma_{n - r + 1}(z - B) = \sigma_n(z-B).\]
    Since the above holds for all $z \in \mathbb{C}$, we have by definition $\Lambda_\eps(B)  \subseteq \Lambda_\eps(B+M)$ for all $\eps > 0$, as desired.
\end{proof}

\section{Moments of the Ginibre Operator Norm} 

\label{sec:ginibretail}
\begin{proof}{Proof of Lemma \ref{lem:CGnp}}
  Begin by observing that
    \begin{equation}
    \label{eq:ginibretail}
    \E[\|G_n\|^p] = p\int_0^2 t^{p-1} \P[\|G_n\|\geq t] dt + p\int_2^\infty t^{p-1} \P[\|G_n\|\geq t] dt \leq  2^p+ p \int_2^\infty t^{p-1} \exp\left\{-n(t-2)^2/2\right\} dt 
    \end{equation}
    where the last inequality used a standard tail bound on $\|G_n\|$ (see for example \cite{davidson2001local}). Now, by Jensen's inequality, for $t\geq 2$ we have $$t^{p-1} = (t-2+2)^{p-1}\leq \frac{1}{2} \left( 2^{p-1}(t-2)^{ p-1}+4^{p-1} \right).$$
    Then, use this inequality and the formula for the absolute moments of the Gaussian distribution to bound the last integral in (\ref{eq:ginibretail}). That is 
    $$\int_2^\infty t^{p-1} \exp\left\{-n(t-2)^2/2\right\} dt \leq 2^{p-2}\cdot  \frac{2^{\frac{p-1}{2}}\Gamma\left(p/2\right)}{2n^{\frac{p-1}{2}} \sqrt{\pi} }+ 4^{p-2}  $$
    Hence 
    $$\E[\|G_n\|^p] \leq 2^p+ \frac{p2^{\frac{p-1}{2}}\Gamma\left(p/2\right)}{2n^{\frac{p-1}{2}} \sqrt{\pi}}+p4^{p-2} = 2^p+ \frac{2^{\frac{p-1}{2}}\Gamma\left(p/2+1\right)}{n^{\frac{p-1}{2}} \sqrt{\pi}}+p4^{p-2}\leq 2^p +\left( \frac{\sqrt{p}}{n^{\frac{p-1}{2p}}}\right)^{p}+ 5^p$$
    Now, since $p\leq \sqrt{n}$ and using the fact that all the terms in the above inequality are positive
    $$
        \E[\|G_n\|^p]^{\frac{1}{p}} \leq 2+ \frac{\sqrt{p}}{n^{\frac{p-1}{2p}}} + 5
    $$
    Since for $x>1$ the function $x^{\frac{x}{x-1}}$ is increasing, and we are assuming that $ p\leq 2n$ we have $p^{\frac{p}{p-1}}\leq (2n)^{\frac{2n}{2n-1}}\leq 4n$. Thus
    $p \leq 4^{\frac{p-1}{p}} n^{\frac{p-1}{p}}$, which implies $\sqrt{p}\leq 2 n^{\frac{p-1}{2p}}$ and concludes the proof.   
\end{proof}

\section{Proof of Theorem \ref{thm:gaps} in the Gaussian Case} \label{sec:gaussian}

We will be terse, as the structure of the proof is identical. When $z \in R$, Theorem \ref{thm:gaussianrealrealsingular} gives
\begin{align}
    \P\left[|\Lambda(A + \gamma \bG_n) \cap D(z,r)| \ge 2 \right] 
    &\le \min_{x>0}\left\{ \frac{nrx}{\gamma} + 4e^2\left(\frac{nr}{2\gamma x}\right)^{4}\right\} \nonumber \\
    &= \frac{5 e^{2/5}}{4} (nr/\gamma)^{8/5}
    \le 2(nr/\gamma)^{8/5}.
\end{align}
Similarly, using Theorem \ref{prop:gaussiansingvalscomplex} for $z \notin \R$,
\begin{align*}
    \P\left[|\Lambda(A + \gamma \bG_n) \cap D(z,r)| \ge 2 \right] 
    &\le \min_{x> 0}\left\{ \frac{\sqrt{7e}n^4}{2\gamma^3}\frac{(9\gamma + \|A\| + |\Re z|)^2 + |\Im z|^2}{|\Im z} (rx)^2 \right. \\
    &\qquad + \left. \frac{4\cdot 7^2e^2n^{14}}{8\gamma^{12}}\left(\frac{(9\gamma + \|A\| + |\Re z|)^2 + |\Im z|^2}{|\Im z|}\right)^{4}(r/x)^8 \right\} \\
    &= \frac{5 (7 e)^{4/5}}{4\cdot 2^{3/5}}\left(\frac{(9\gamma + \|A\| + |\Re z|)^2 + |\Im z|^2}{|\Im z|}\right)^{8/5} n^6 r^{16/5}\gamma^{-24/5} \\
    &\le 9\left(\frac{(9\gamma + \|A\| + |\Re z|)^2 + |\Im z|^2}{|\Im z|}\right)^{8/5} n^6 r^{16/5}\gamma^{-24/5}
\end{align*}
Using the same net as in the original proof, and taking $R \defeq \|A\| + 4\gamma$,
\begin{align*}
    \P\left[\gap(A + \gamma\bG_n) \le s \right]
    &\le \sum_{z \in \calN_{2s}^{\R}}\P\left[|\Lambda(A + \gamma\bG_n) \cap D(z,3s/2)| \ge 2 \right] 
        + \sum_{z \in \calN_{\delta}^{\R}}P\left[|\Lambda(A + \gamma\bG_n) \cap D(z,\sqrt{2}\delta)| \ge 2 \right] \\
        &\qquad + \sum_{z \in \calN_{\delta,s}^{\C}}\P\left[\Lambda(A + \gamma\bG_n) \cap D(z,\sqrt{5/4} s)| \ge 2 \right]
        + \P\left[\|A + \gamma\bG_n\| \ge R \right] \\
    &\le \frac{3(\|A\| + 4\gamma)}{2s}\cdot 2(3ns/2\gamma)^{8/5} + \frac{3(\|A\| + 4\gamma)}{2\delta}\cdot 2(\sqrt 2 n\delta/\gamma)^{8/5} \\
        & \qquad + 6\left(\frac{\|A\| + 4\gamma}{s}\right)^2 \cdot 9\left(\frac{4(\|A\| + 6.5\gamma)^2}{\delta}\right)^{8/5} n^6 (\sqrt{5/4}s)^{16/5}\gamma^{-24/5} 
        + e^{-2n} \\
    &\le 6\left(\|A\| + 4\gamma\right)(n/\gamma)^{8/5}s^{3/5} + 6\left(\|A\| + 4\gamma \right)(n/\gamma)^{8/5}\delta^{3/5} \\
        &\qquad + 800 \left(\|A\| + 6.5\gamma\right)^{26/5} n^6 s^{6/5}\delta^{-8/5}\gamma^{-24/5} + e^{-2n}.
\end{align*}
Finally, optimizing in $\delta$ using the same argument as the main proof, and $\gamma < 1$,
\begin{align*}
    \P\left[\gap(A + \gamma\bG_n) \le s \right] 
    &\le 6\left(\|A\| + 4\gamma\right)(n/\gamma)^{8/5}s^{3/5} \\
        &\qquad + 2\left(6\left(\|A\| + 4\gamma \right)(n/\gamma)^{8/5}\right)^{8/11}\left(800 \left(\|A\| + 6.5\gamma\right)^{26/5} n^6 \gamma^{-24/5}\right)^{3/11}s^{18/55} + e^{-2n} \\
    &\le 6(\|A\| + 4\gamma)(n/\gamma)^{8/5}s^{3/5} + 7 (\|A\| + 6.5\gamma)^{118/55}n^{64/55}n^{18/11}\gamma^{-136/55}s^{18/55} + e^{-2n} \\
    &\le 15 \left(\|A\| + 7\right)^3 n^3\gamma^{-5/2}s^{1/3} + e^{-2n}.
\end{align*}
As in the non-Gaussian case, we separately state a tail bound for $\Im_{\min}$:
\begin{equation}
    \label{eq:im-min-standalone-gaussian}
    \P\left[\Im_{\min}(A + \gamma \bG_n) \le \delta\right] \le 6\left(\|A\| + 4\gamma\right)(n/\gamma)^{8/5}\delta^{3/5}.
\end{equation}

\end{document}